\documentclass[11pt]{amsart}

\usepackage[adobe-utopia]{mathdesign}

\usepackage{geometry}               
\usepackage{amsrefs} 				
\usepackage[T1]{fontenc}
\usepackage{amssymb, color}
\usepackage[toc,page]{appendix}

\usepackage{fancyhdr}
\usepackage{lastpage} 

\usepackage{xypic}
\usepackage{enumerate} 

\usepackage{fancyvrb}	
\usepackage[usenames,dvipsnames]{xcolor}
\fvset{frame=single,framesep=1mm,fontfamily=tt,fontsize=\scriptsize,numbers=left,framerule=.3mm,numbersep=1mm,commandchars=\\\{\}}

\usepackage{graphicx}
\usepackage[colorlinks = true, citecolor = red]{hyperref}
\usepackage[capitalise, noabbrev]{cleveref}

\newcommand*{\myTagFormat}[2]{(\cref{#1})($#2$)}

\def\apg{a_{\tild P} - a_{\tild G}}
\def\A{\mathbb A}
\def\C{\mathbb C}
\def\Q{\mathbb Q}
\def\R{\mathbb R}

\def\AAA{\mathcal A}	
\def\CCC{\mathcal C}
\def\DDD{\mathcal D}
\def\FFF{\mathcal F}
\def\LLL{\mathcal L}
\def\MMM{\mathfrak M}	
\def\PPP{\mathcal P}
\def\O{\mathcal O}
\def\o{\scalebox{0.8}{$\scriptstyle\mathcal{O}$}}
\def\myshift{\mathcal S}	
\def\UUU{\mathcal U}
\def\aaa{\mathfrak a}
\def\d{\text d}
\def\K{\textbf K}
\def\Ad{\operatorname{Ad}}
\def\Aut{\operatorname{Aut}}
\def\bs{\setminus} 			
\def\det{\operatorname{det}}
\def\dim{\operatorname{dim}}
\def\Gal{\operatorname{Gal}}
\def\GL{\operatorname{GL}}
\def\geom{\text{geom}}

\def\Id{\operatorname{Id}}
\def\Ind{\operatorname{Ind}}
\def\l{\ell}
\def\Lone{L^1}
\def\Ltwo{L^2}
\def\lmod#1{\left\lvert #1 \right\vert} 
\def\Lieg{\mathfrak g}
\def\mod#1{\vert #1 \vert} 
\def\norm#1{\Vert #1 \Vert} 
\def\RRR{\mathcal R}
\def\Re{\operatorname{Re}}
\def\Res{\operatorname{Res}}
\def\se{\subseteq}
\def\sprod#1#2{\left\langle #1 , #2 \right\rangle}  
\def\tild{\widetilde}
\def\trace{\operatorname{trace}}
\def\Vol{\operatorname{Vol}}

\newtheorem{theorem}{Theorem}[section]

\newtheorem{assumption}[theorem]{Assumption}
\newtheorem{lemma}[theorem]{Lemma}

\theoremstyle{remark}
\newtheorem{remark}[theorem]{Remark}

\begin{document}

\title{Absolute Convergence of the Twisted Trace Formula}
\author{ Abhishek Parab}
\address{Purdue University}
\email{aparab@purdue.edu}

\maketitle
\setcounter{tocdepth}{1}		

\begin{tabular} {ll}
\textbf{Author} : & Abhishek Parab \\ & Purdue University \\ & 150 N. University St, \\& West Lafayette, IN 47906 \\
\textbf{Email} : & aparab@purdue.edu \\
\textbf{Telephone} : & +1-765-418-3456 \\
\textbf{ORCID} : &  \href{https://orcid.org/0000-0002-4824-6419}{0000-0002-4824-6419} \\[2em]
\end{tabular}

\textbf{Title} : Absolute Convergence of the Twisted Trace Formula \\

\textbf{Abstract} : We show that the distributions occurring in the geometric and spectral side of the twisted Arthur-Selberg trace formula extend to non-compactly supported test functions. The geometric assertion is modulo a hypothesis on root systems proven when the group is split. The result extends the work of Finis-Lapid (and M\"uller, spectral side) to the twisted setting. We use the absolute convergence to give a geometric interpretation of sums of residues of certain Rankin-Selberg L-functions. \\[1em]

\begin{tabular}{ll}
\textbf{Keywords} : & twisted trace formula, Arthur-Selberg trace formula, automorphic forms, \\ & Rankin-Selberg $L$-functions, Beyond Endoscopy \\
\textbf{MSC} : & 11F72 (primary), 11F70, 22E55 (secondary). \\[2em]
\end{tabular}

\textbf{Acknowledgements} : The author would like to thank his advisor Prof. Freydoon Shahidi for everything, and more. He also thanks J. Getz, E. Lapid, D.B. McReynolds, N. Miller, C.-P. Mok, P. Solapurkar and S. Yasuda for useful discussions and encouragement. The proof of \cref{majer} is due to P. Majer via Mathoverflow. The application was suggested by J. Getz and we thank him also for insightful discussions. We thank the anonymous referee and the editor for their useful comments. The author was supported by National Science Foundation grant DMS-1162299 through Prof. Shahidi. 

\newpage

\tableofcontents


\section{Introduction}

Understanding the automorphic representations of a connected reductive group over a global field has been a central problem in the theory of automorphic forms and the Arthur-Selberg trace formula has been an indispensable tool in doing so. The trace formula was conceived by Selberg \cite{MR0088511} to show the existence of Maass forms via a Weyl law, and vastly generalized by Arthur to any connected reductive group $G$ over a number field $F$. It is an identity of two distributions on $G(\A)$ where $\A = \A_F$ is the ring of adeles of $F$ viz., 
\[ J^G_{\text{spec}}(f) = J^G_{\text{geom}}(f), \]
$f$ being a smooth function on $G(\A)$ of compact support. The spectral side is a sum-integral over automorphic representations of $G$ and the geometric side contains weighted orbital integrals. It was realized early on that an appropriate `twisted' trace formula developed for a connected component $\tild G$ of a reductive group $G$ would be useful in proving the `endoscopic' cases of Langlands' Functoriality conjectures. If $\theta$ is an $F$-automorphism of $G$ of finite order, we can form the reductive group $G \rtimes \langle \theta \rangle$ of which $\tild G = G \rtimes \theta$ is a connected component. The trace formula for $\tild G$ or the twisted trace formula was developed in the lectures given at the Friday Morning seminar at IAS organized by Clozel, Labesse and Langlands \cite{CLL} and has been exposed  and improved upon in the book \cite{LW}. The twisted trace formula has also been instrumental in proving the cyclic (solvable) base change in the book of Arthur and Clozel \cite{AC}. 

If $\phi : {}^LH \to {}^LG$ is an $L$-homomorphism between the $L$-groups of quasisplit connected reductive groups $H$ and $G$, Langlands' Functoriality predicts a transfer of automorphic representations of $H$ to $G$. Among these are the $L$-homomorphisms arising out of endoscopic groups of which the classical groups (orthogonal and symplectic) are prototypical examples when $G = GL(n)$ and $\theta(g) = {}^t\!g^{-1}$. Arthur \cite{MR3135650} proved functoriality in this case and Mok \cite{MR3338302} extended it to unitary groups. This had been conditional on the Fundamental Lemma which was resolved by Ng\^o \cite{MR2653248} and the stabilization of the twisted trace formula by M{\oe}glin-Waldspurger \cite{MWTFL1, MWTFL2}. Arthur proved the geometric side of the trace formula converges for $f \in \CCC_c^\infty(G(\A))$ but he 
 deftly didn't make use of the convergence of the spectral side. The convergence was proven in \cite{MS} for $G = GL(n)$ and later by Finis, Lapid and M{\"u}ller \cite{FL11a, FLM} for general $G$. Our work extends their result to the twisted trace formula. 

Finis-Lapid have proven the absolute convergence of the spectral and geometric sides for more general test functions than those of compact support whose extension to the twisted setting is the main theme of this paper. Let $\K$ be a ``good'' maximal compact subgroup of $G(\A)$ and $K$ be an open compact subgroup of the finite adeles $G(\A_f)$ in $\K$. They consider a class $\CCC(G(\A), K)$ of test functions $f$ on $G(\A)$ which are right $K$-invariant at non-Archimedean places and at the Archimedean places satisfy, $\norm{f * X}_{L^1(G(\A))} < \infty$ for every  $X \in \UUU(\Lieg_\C)$, the universal enveloping algebra of the Lie algebra $\Lieg_\C$  acting as differential operators. For such test functions they prove the convergence on the spectral side \cite{FL11b, FLM} and also the geometric side \cite{FL11a, FL16}, thus constructing a trace formula for a broader class of test functions. In the case when $\tild G$ is a component of $G \rtimes \langle \theta \rangle$, the space $\CCC(\tild G(\A), K)$ can be defined similarly by considering the action of $X$ on a smooth function $f$ defined on $\tild G(\A)$. The main result of this article is to derive the convergence and hence the continuity of the distributions occurring in the twisted Arthur-Selberg trace formula with \cref{geom} and \cref{spectral} being the statements for the geometric and spectral sides respectively. 

The proof of convergence in the geometric side involves estimating the sums over certain twisted Bruhat cells and using the slow decay of intertwining operators and is carried out in Sections~\ref{section:geom}~and~\ref{section:3.2tild}. The main steps follow \cite{FL11a, FL16}  except that the twisted equivalent of the crucial Lemma 2.2 of \cite{FL11a} does not hold. The corresponding modification is discussed in \cref{section:RCL} as the Root Cone Lemma. The convergence on the geometric side is modulo this geometric lemma but we prove it completely when $G$ is split and also for cyclic base change. It is a lemma about root systems involving automorphisms of the Dynkin diagram of $G$ so depends only on the semisimple part of $G$. We reduce it to split simple groups and then prove all cases in the Cartan-Killing classification. A crucial step in the proof when $G$ is of type $A_{n-1}$ was shown to us by P. Majer as an answer to a question \cite{mathoverflow} on MathOverflow. The lemma for $E_6$ type is proven using the Mathematical software {\tt{SageMath}} \cite{sagemath}. 

The convergence results usually involve a parameter $T$ in a finite dimensional vector space $\aaa_0$ (see \cref{section:notation} for definition) chosen sufficiently away from the origin. In our setting however, $T$ is allowed to vary only along a line but it is enough to ensure the convergence. Since the distributions involved are polynomials in $T$, they can be extended to any $T \in \aaa_0$ and in particular to the special point $T_0$ that makes the distributions independent of the chosen minimal parabolic subgroup, see \cite{MR625344}*{Lemma 1.1}. For applications to limit multiplicities it is essential to keep track on the dependence the compact open subgroup $K$ of $G(\A_f)$ in proving the bound on the seminorm $\mu$ but the non-twisted bounds work verbatim. 

\Cref{section:spec} is devoted to proving the spectral side. The convergence of the spectral side in the non-twisted setting involved estimating the derivatives of certain intertwining operators appearing in the spectral expansion and has been the main result of \cite{FL11b}. The crucial difference on the spectral side of the twisted trace formula is that the trace here is a composition of operators on different spaces. We introduce a unitary shift operator which converts the twisted trace formula to the usual one and invoke the estimates in [ibid]. 

In \cref{section:application} we give an application of the continuity of the spectral side towards giving a geometric interpretation for a sum of residues of certain Rankin-Selberg $L$-functions, under the hypothesis of \cref{assump}. This was suggested to us by J. Getz and has conjectural applications towards Langlands' `Beyond Endoscopy' proposal. Another possible application of the main results of this paper is the Weyl Law for self-dual automorphic representations in the style of \cite{LM} which would explain the endoscopic classification of classical groups in a more quantitative way. This is currently a work in progress. 


\section{Notation} \label{section:notation}

	We will follow the notations of Labesse and Waldspurger \cite{LW} rather than of Arthur. Throughout, $G$ will denote a connected reductive group over $\Q$ and $\tild G$, a twisted $G$-space [ibid., Chapter 2] such that $\tild G(\Q)$ is nonempty. Thus $\tild G$ is a left-$G$-torsor equipped with a map 
	\[ \Ad : \tild G \to \Aut(G) \]
	which for $x \in G, \delta \in \tild G$ satisfies,
	\[ \Ad(x\delta) = \Ad(x) \circ \Ad(\delta). \]
	We will assume that $G$ and $\tild G$ are the components of a reductive algebraic group. This is somewhat more restrictive than \cite{LW} but suffices for most applications of the twisted trace formula. All algebraic subgroups of $G$ will be implicitly assumed to be defined over the rational numbers. 
	We define a right-$G$-action on $\tild G$ via 
	\[ \delta x = \theta(x) \delta, \qquad \text{with } \theta = \Ad(\delta). \]
	Any element $y \in \tild G(\A)$ can be written (not uniquely) as $y = x \delta$ with $x \in G(\A)$ and $\delta \in \tild G(\Q)$. 
		
	Throughout this paper, we fix a minimal parabolic subgroup $P_0$ of $G$ defined over $\Q$ with Levi decomposition $P_0 = M_0 \ltimes N_0$ and a maximal compact subgroup $\K_G = \K = \K_\infty \K_f$ which is admissible relative to $M_0$ in the sense of \cite{MR625344}*{\S 1}. Thus we have the Iwasawa decomposition 
	\[ G(\A) = P_0(\A) \K_G. \]

	We will denote the (finite) Weyl group of $(G,T_0)$ by $W^G$ or $W$. Note that $M_0$ is the centralizer of a maximal torus which we denote by $T_0$. The Lie algebra of $G$ will be denoted by $\Lieg$ and the universal enveloping algebra of its complexification by $\mathcal U(\Lieg_\C)$. 
	
	Let $\LLL^G(M_0) = \LLL^G = \LLL$ denote the set of Levi subgroups containing $M_0$, i.e., the (finite) set of centralizers of subtori of $T_0$. For $M \in \LLL$, we have the following notation. 
	
	\begin{itemize}
		\item We shall denote by $\LLL^G(M), \FFF^G(M), \PPP^G(M)$, the (finite) set of Levi subgroups containing $M$, parabolic subgroups containing $M$ and parabolic subgroups with Levi part $M$ respectively. Whenever clear from the context, we shall ignore the superscript $G$ or replace it with a reductive subgroup of $G$. The notations $\LLL, \FFF$ (and $\PPP$) will denote the Levi (resp. parabolic) subgroups containing $M_0$ (resp. $P_0$). 
		
		\item The Weyl group $W^G(M) = W(M)= N_{G(\Q)}(M)/M$ can be identified with a subgroup of $W$. 
		\item $T_M$ is the split part of the identity component of the center of $M$ and $A_M = A_0 \cap T_M(\R)$.
		 
		\item The real vector space $\aaa_M^*$ is spanned by the lattice $X^*(M)$ of rational characters of $M$ and $\aaa_{M,\C}^*$ is its complexification. The dual space $\aaa_M$ spanned by the co-characters of $T_M$ is the Lie algebra of $A_M$. Denote by $a_M$ the dimension of each. 
		
		\item The map $H_M : M(\A) \to \aaa_M$ is the homomorphism given by $\sprod{\chi}{H_M(m)} = \log \mod{\chi(m)}_{\A^*}$, for any $\chi \in X^*(M)$. Its kernel is $M(\A)^1$. 
		
		\item $X_M =  A_M M(\Q) \bs M(\A). $
		
		\item The Weyl group $W(M)$ acts on $\PPP(M)$ and $\FFF(M)$ by conjugation; $w.P = n_w P n_w^{-1}$. 
		
		\item $\RRR_M$ is the set of reduced roots of $T_M$ on $\Lieg$ and for every root $\alpha \in \RRR_M$, $\alpha^\vee$ denotes the corresponding co-root. It will be abbreviated $\RRR_0$ when $M = M_0$. 
		
		\item For $w \in W, Q(w)$ denotes the smallest standard parabolic subgroup containing a representative $n_w$ of $w$. 

	\end{itemize}
	
	In particular, we have the above notation for $M=G$. For $P \in \PPP(M)$, we use the following additional notation. 
	\begin{itemize}
	
		\item $N_P$ is the unipotent radical of $P$ and $M_P$ is the unique element $L \in \LLL(M)$ with $P \in \PPP(L)$. 
		
		\item $A_P = A_{M_P}; \ \aaa_P = \aaa_M;\ a_P = \text{ dim } \aaa_P$.
				
		\item For a point $Z\in \aaa_0, Z_P$ denotes the projection of $Z$ onto $\aaa_P$. 
		
		\item The map $H_P : G(\A) \to \aaa_P$ is the extension of $H_M$ to a left $N_P(\A)-$ and right $\K$-invariant map. 
		
		\item $\Delta_P$ (resp. $\hat{\Delta}_P$) is the subset of simple roots (resp. simple weights) of $P$, which is a basis for $(\aaa_P^G)^*$. 
		
		\item $X_P = A_P N_P(\A)M_P(\Q) \bs G(\A), \ Y_P = A_G P(\Q) \bs G(\A)$. 
		
		\item Denote by $\xi_P$ the sum of roots in $\Delta_0^P$. More generally if $Q$ is a parabolic subgroup containing $P$ then denote by $\xi_P^Q$ the sum of roots in $\Delta_P^Q$ which is the set $\Delta_0^Q \setminus \Delta_0^P$. 

		\item For $X \in \aaa_P, d_P(X) = \inf_{\alpha \in \Delta_P} \alpha(X)$. Indeed, it denotes the distance of $X$ from the `walls'. 
		
		\item The function $\tau_P^Q$ is the characteristic function of the open cone
			\[ \tau_P^Q(H) = 1 \Leftrightarrow \alpha(H) > 0 \quad \forall \alpha \in \Delta_P^Q. \]
			The function $\hat\tau_P^Q$ is defined analogously, see \cite{MR518111}*{Section 1}. 
		
		\item The Killing form induces an invariant inner product and a Euclidean structure on $\aaa_P^G$. $\Vol(\Delta_P^\vee)$ is the volume of the parallelopiped in $\aaa_P^G$ whose sides are roots in $\Delta_P^\vee$. For $\Lambda \in i \aaa_M^*$ regular, define \footnote{Arthur uses $\theta_P(\Lambda)$ for the inverse of $\epsilon_P(\Lambda)$ but following \cite{LW}, we reserve the symbol $\theta$ exclusively for the automorphism on $G$. }
	\begin{equation} \label{epsilon}
		\epsilon_P (\Lambda) = \Vol(\Delta_P^\vee) \prod_{\alpha \in \Delta_P} \sprod{\Lambda}{\alpha^\vee}^{-1}.
	\end{equation}
		We define $\hat\epsilon_P(\Lambda)$ by replacing $\Delta_P$ by $\hat \Delta_P$ and similar sets $\epsilon_P^Q, \hat\epsilon_P^Q$ whenever $P \subseteq Q$. 
		
		\item $\AAA(X_P)$ is the space of automorphic forms on $X_P$ (cf. \cite{MW}*{\S I.2.17} and \cite{MR546598}*{\S 4}). For an automorphic representation $\sigma$ of $M$, the space $\AAA(X_P, \sigma)$ is the space of automorphic forms $\Phi$ over $X_P$ such that for every $x \in G(\A)$, the function 
		\[ m \mapsto \Phi(mx), \quad \text{for } m \in M(\A) \]
		is an automorphic form in the $\sigma-$isotypical space of $\Ltwo_{\text{disc}}(X_M)$. 
		
		\item The spaces $\AAA(X_P), \AAA_{\text{disc}}(X_P)$ and $\AAA_{\text{cusp}}(X_P)$ (resp. square-integrable and cuspidal  forms) are pre-Hilbert spaces with respect to the inner product 
		\[ \sprod{\Phi}{\Psi}_P = \int_{X_P} \Phi(x) \Psi(x) \d x. \]
		We denote by $\overline{\AAA}(X_P)$ the Hilbert space completion of $\AAA(X_P)$. 
	\end{itemize}
	
	Now we define some objects related to $\tild G$. Some of the above notation needs to be modified appropriately for such objects. 
	\begin{itemize}
	
		\item We fix once and for all an element $\delta_0 \in \tild G(\Q)$ such that the automorphism $\theta_0 = \Ad(\delta_0)$ preserves $P_0$ and $M_0$. Such an element is uniquely determined modulo conjugation by $M_0(\Q)$. 
		
		\item A parabolic subset $\tild P$ of $\tild G$ is the normalizer in $\tild G$ of a parabolic subgroup $P$ of $G$ such that $\tild P(\Q) \neq \emptyset$. 
		
		\item $\tild M_P$ is the Levi subset of $\tild P$ if there is a Levi decomposition $\tild P = \tild{M_P}N_P$, where $N_P$ is the unipotent radical of $P$, which is invariant under $\Ad(\delta), \ \delta \in \tild P(\Q)$. 
		
		\item $\tild{P_0} = P_0.\delta_0$ and any parabolic subset containing $\tild{P_0}$ is called standard, denoted by $\PPP^{\tild G}(M_0)$ and abbreviated as $\tild \PPP(M_0)$ or simply $\tild \PPP$. The sets $\LLL^{\tild G}(M)$ and $\FFF^{\tild G}(M)$ are defined similarly for any $M \in \LLL$. 
		
		\item When $P = P_0$, we extend the map $H_0 = H_{P_0}$ to $\tild G(\A)$ by $H_0(x\delta_0) = H_0(x)$.  This is well-defined because $G(\Q)$ is in the kernel of $H_P$. 
		
		\item The automorphism $\theta_0$ on $G$ induces a linear map, also denoted by $\theta_0$, on $\aaa_M$ via the action on co-characters. The space $\aaa_{\tild P} = \aaa_{\tild M}$ is the set of vectors of $\aaa_M$ fixed under this automorphism. In particular, we can identify $\aaa_{\tild P}$ as a subset of $\aaa_P$. 
		
		\item As before, $a_{\tild P}$ will denote the dimension of $\aaa_{\tild P}$. 
		
		\item An inclusion $\tild P \subset \tild Q$ of parabolic subsets gives $\aaa_{\tild Q} \subset \aaa_{\tild P}$ and a canonical decomposition 
		\[ \aaa_{\tild P} = \aaa_{\tild Q} \oplus \aaa_{\tild P}^{\tild Q}. \]
		
		\item A root $\alpha \in \Delta_P^Q$  induces a linear form $\tild \alpha$ on $\aaa_{\tild P}^{\tild Q}$ by averaging 
		\[ \displaystyle \frac{1}{l} \sum_{r=0}^{l-1} \theta_0^r(\alpha), \]
		where $l$ is the order of the automorphism $\theta_0$. Denote by $\Delta_{\tild P}^{\tild Q}$ the set of such orbits. Analogously we define the set $\hat{\Delta}_{\tild P}^{\tild Q}$ of orbits of weights $\varpi \in \hat{\Delta}_P^Q$. 
		
		\item The characteristic functions $\tau_{\tilde P}^{\tilde Q}$ and $\hat\tau_{\tilde P}^{\tilde Q}$ are defined by replacing $\Delta_P^Q$ and $\hat\Delta_P^Q$ in the definition of $\tau_P^Q$ and $\hat\tau_P^Q$ with $\Delta_{\tilde P}^{\tilde Q}$ and $\hat\Delta_{\tilde P}^{\tilde Q}$ respectively. When $Q=G$ we ignore the superscript. 
		
		\item The (inclusion-reversing) bijection $P \mapsto \Delta_{P_0}^P$ between standard parabolic subgroups and subsets of the simple roots, in the twisted case becomes a bijection $\tild P \mapsto \Delta_{\tild P}^{\tild P_0}$ of corresponding sets. 
		
		\item For a standard parabolic subgroup $P$, we define subgroups $P^- \subseteq P \subseteq P^+$ as follows. $P^-$ is the standard parabolic subgroup whose Levi has, for simple roots, those $\alpha \in \Delta_0^P$ such that the orbit of $\alpha$ under $\theta_0$ is contained in $\Delta_0^P$. Likewise $P^+$ is the standard parabolic subgroup whose Levi has, for simple roots, elements of orbits of $\Delta_0^P$ under $\theta_0$. 
		
		\item Similar to the non-twisted case, we define for $\Lambda \in i\aaa_{\tild M}^*$ regular, 
		\[ \epsilon_{\tild P}(\Lambda) = \Vol(\Delta_{\tild P}^\vee) \prod_{\alpha \in \Delta_{\tild P}} \sprod{\Lambda}{\alpha^\vee}^{-1}. \]
		
		\item The Weyl set $W^{\tild G} = \tild W$ is the quotient of the normalizer of $M_0$ in $\tild G$ by $M_0$. Indeed,
		\[ \tild W = W \rtimes \theta_0 \]
		and the representatives of the Weyl set can be chosen as $n_w \delta_0$ where $n_w$ are representatives of $W$. 
		
		\item Throughout the paper, we shall fix a unitary character $\omega$ of $G(\A)$ which is trivial on $A_G G(\Q)$. 
		
		\item A (twisted) representation of $(\tild G, \omega)$ is a representation $\pi$ of $G$ on a vector space $V$ along with an invertible endomorphism 
		\[ \tild{\pi}(\delta, \omega) \in \GL(V), \quad \delta \in \tild G(\Q) \]
		satisfying for every $x,y \in G$ and $\delta \in \tild G(\Q)$,
		\[ \tild \pi (x\delta y, \omega) = \pi(x) \ \tild \pi(\delta, \omega) \ (\pi \otimes \omega)(y). \]
		\item Since $G$ is unimodular, the measure on $G(\A)$ induces a measure on $\tild G(\A)$ via 
		\[ \int_{\tild G(\A)} h(y) \d y = \int_{G(\A)} h(x\delta) \d x, \quad \delta \in G(\Q). \]
		
		\item For an integrable function $f$ on $\tild G(\A)$, define 
			\[ f^1(y) = \int_{A_G} f(zy) \d z. \]
		
	\end{itemize}
	
	\begin{remark} \label{rem:defQ}
		By restriction of scalars there is no loss of generality to consider groups defined over $\Q$ and henceforth we assume so.  
	\end{remark}
	

\section{Preliminaries}

	There is an action of $\tild G(\A)$ on the homogeneous space $X_G = A_G G(\Q) \bs G(\A)$ given by 
	\[ (y, \dot{x} ) \mapsto \dot x * y = \delta^{-1}x y \]
	where $y \in \tild G(\A), x \in G(\A) $ is a representative of $\dot x$ and $\delta$ is any element of $\tild G(\Q)$. When there is no confusion, we shall denote $\dot x$ by its representative $x$. 
	
	\subsection{The space $\CCC(\tild G(\A), K)$}
	
	Fix a compact open subgroup $K$ of $G(\A_f)$ in $\K_f$, where $\A_f$ is the ring of finite adeles. The right action of $G(\A)$ on $\tild G(\A)$ restricts to that of $K$. For a smooth function $h$ on $\tild G(\R)$ and $X \in \UUU(\Lieg)$,  we define the smooth function $h*X$ on $\tild G(\R)$ by
	\[ (h*X)(y) = \frac{\d}{\d t} \ h(y \exp {tX})\bigg|_{t=0}. \]
	We extend this action to smooth functions on $\tild G(\A)$ by ignoring the non-Archimedean component. 
	
	Define $\CCC(\tild G(\A), K)$ to be the space of smooth functions $h$ on $\tild G(\A)$ which are right $K$-invariant and which satisfy 
	\[ \norm{h*X}_{\Lone(\tild G(\A))} < \infty, \]
	for any $X \in \UUU(\Lieg)$. The topology induced from the seminorms $\norm{f*X}_{\Lone(\tild G(\A))}$ makes $\CCC(\tild G(\A), K)$ into a Frechet space (i.e., complete, metrizable and locally convex), see \cite{MR0225131}*{Chapter 10}. Sometimes we will abbreviate $\norm{h}_1$ for the $\Lone$-norm of $h$. The following lemma is proved in \cite{MR0225131}*{Proposition 7.7}. 
	\begin{lemma} 
		A linear form $J$ on a locally convex space $E$ is continuous if and only if there is a continuous seminorm $\mu$ on $E$ such that for every $f \in E$, 
		\[ J(f) \leq \mu(f). \]
	\end{lemma}
	
	Note that in \cites{FLM, FL16}, Finis, Lapid and M\"{u}ller prove the continuity of the usual (non-twisted) trace formula with the analogous space $\CCC(G(\A),K)$. Indeed, we have a correspondence between the two spaces: 
	\begin{lemma} \label{bijectionlemma}
		For $f \in \CCC(G(\A), K)$ and $\delta \in \tild G(\Q)$ define the function $(L_\delta f)(y) = f(\delta^{-1}y)$ on $\CCC(\tild G(\A),K)$ then this map is a bijection between the two spaces with inverse $L_{\delta^{-1}}$. Moreover, $\norm{f} = \norm{L_\delta f}$ and $L_\delta(f*X) = L_\delta(f)*X$ for any $X \in \UUU(\Lieg_\C)$. 
	\end{lemma}
	
	The bijection is obvious. The equality of $\Lone$ norms is a consequence of the definition of measure on the twisted space $\tild G(\A)$. 

\subsection{Reduction Theory}

	For $Q \in \PPP$, $T_1, T \in \aaa_0$, we define the Siegel set $\mathcal S^Q_{P_0}(T_1, T)$ consisting of $x = pak \in G(\A)$ such that $k \in \K, p \in \omega$, a fixed compact subset of $M_0(\A)^1 N_0(\A)$ and $a \in A_0$ satisfying 
	\[ \tau_{P_0}^Q(H_0(a)-T_1) = 1; \quad \hat{\tau}^Q_{P_0}(T - H_0(a))=1. \]
	Then we have the partition lemma of Langlands that for any $x \in G(\A)$ and $-T_1, T$ sufficiently regular, i.e., $d_0(T) \geq c, d_0(-T_1) \geq c_1$ for fixed positive constants $c, c_1$,
	\begin{equation} \label{partition}
		\displaystyle \sum_{\substack{Q : \\ P_0 \se Q \se P}} \sum_{\delta \in P(\Q) \bs Q(\Q)} F^Q_{P_0}(\delta x, T) \tau_Q^P(H_0(\delta x) - T) = 1,
	\end{equation}
	where $F^Q_{P_0}(\circ, T)$ is the characteristic function of the set $Q(\Q) \mathcal S^Q_{P_0}(T_1, T)$. In particular, for $P=G$ we have 
	\[ G(\Q)\ \mathcal S^G_{P_0}(T_1, T) = G(\A). \]	
	Throughout the paper we fix such $T_1 \in \aaa_0$. We remark that this version of the partition lemma is due to Arthur \cite{MR518111}*{Lemma 6.4} who attributes it to Langlands \cite{MR0579181}*{Lemma 2.5}. 

\subsection{The operator $\rho$}

	The usual right regular action $\rho$ of $G(\A)$ on $\Ltwo(X_G)$, which is given by
	\[ (\rho(g)\Phi)(x) = \Phi(x g) \]
	extends to a twisted representation $\tild \rho$ of $\tild G(\A)$: 
	\[ (\tild \rho(y , \omega)\Phi)(x) = (\omega \Phi)(\dot x*y) = (\omega \Phi)(\delta^{-1}xy). \]
	
	The representation of $G(\A)$ on $\Ltwo(X_G)$ decomposes into a discrete spectrum and a continuous spectrum: 
	 \[ \Ltwo(X_G) = \Ltwo_\text{disc}(X_G) \oplus \Ltwo_\text{cont}(X_G). \]
	 We shall denote by $\Pi_\text{disc}(\tild G, \omega)$ the equivalence classes of automorphic representations $\pi \in \Pi_\text{disc}(G)$ which extend to a twisted representation $\tild \pi$ of $\tild G(\A)$. They are precisely those satisfying $\pi \sim \pi \circ \theta$ \cite{LW}*{Lemme 2.3.2}. 
	
	Fix $P \in \PPP(M)$. The induced representation of $G(\A)$ on $\AAA(X_P)$ is given by 
	\[ (\rho_{P,\nu}(g)\Phi)(x) = \Phi(x g) \ \exp \sprod{\nu + \rho_P}{H_P(xg)-H_P(x)}. \]
	It is isomorphic to $\Ind_{P(\A)}^{G(\A)} \Ltwo_{\text{disc}}(X_M) \otimes \exp \sprod{\nu}{H_M(\circ)} $. 
	
	A compactly supported smooth $\K$-invariant function $h$ on $G(\A)$ defines an operator $\rho_{P,\nu}(h)$ on $\overline A(X_P)$ by
	\[ \rho_{P,\nu}(h)(\Phi) = \int_{X_P} h(x) \rho_{P,\nu}(x)(\Phi) \d x \] 
	whose image lies in the subspace of smooth $\K$-invariant functions. 

	We now define the twisted analog of $\rho$. Assume $P \in \PPP$ and $\delta \in \tild G(\Q)$. Denote by $Q$ the parabolic subgroup obtained by conjugation by $\delta$, i.e., $Q = \delta P \delta^{-1} = \theta(P)$ where $\theta = \Ad(\delta)$. Let $\sigma$ be an automorphic representation of $M$. 
	
	An element $y \in \tild G(\A)$ defines an operator for $\nu \in \aaa_{P,\C}^*$ ,
	\[ \tild \rho_{P,\sigma, \nu}(\delta,y,\omega):\overline{\AAA}(X_P,\sigma) \to \overline \AAA(X_Q, \sigma\circ\theta^{-1}) \]
	by
	\[ (\tild \rho_{P,\sigma, \nu}(\delta,y,\omega)\Phi)(x) = \exp \sprod{\theta(\nu+\rho_P)}{H_Q(x)} (\omega \Phi)(\delta^{-1}xy) \exp \sprod{\nu+\rho_P}{H_P(\delta^{-1}xy)}.  \]
	Likewise, by integrating against a smooth function $f \in \CCC(\tild G(\A), K)$, we define the operator 
	\[ \tild \rho_{P, \sigma, \nu}(\delta, f, \omega) = \int_{\tild G(\A)} f(y) \tild \rho_{P,\sigma, \nu}(\delta,y,\omega) \d y \]
	from the space $\AAA(X_P)$ to $\AAA(X_Q)$. Hopefully the notations of $\rho_{P, \nu}$ and $\tild \rho_{P, \sigma, \nu}$ for the induced representations of $G(\A)$ and $\tild G(\A)$ will not be confused with $\rho_P$, which is half of the sum of positive roots of $P$. \footnote{Our notation differs from \cite{LW} wherein $\tild \rho(f, \omega)$ stands for $\rho(\delta, f, \omega)$ for some $\delta \in \tild G(\Q)$. }

	
\section{The geometric side} \label{section:geom}

	\subsection{Statement of the geometric continuity}

		Assume that $f \in \CCC^\infty_c(\tild G(\A))$ and denote 
		\[ f^1(y) = \int_{A_G} f(yz) \d z \]
		where as usual, $A_G$ is the set of real points of the maximal $\Q$-split torus in the center of $G$ or equivalently, the kernel of the map $H_G : G(\A) \to \aaa_G$. 
		
		An element $\delta \in \tild G(\Q)$ has a Jordan decomposition 
		\[ \delta = s_\delta n_\delta = n_\delta s_\delta \]
		where $n_\delta$ is a unipotent element of $G(\Q)$ and $s_\delta \in \tild G(\Q)$ is quasi-semisimple in that the automorphism $\Ad(s_\delta)$ induced on the derived group $G_{\text{der}}$ is semisimple. Two elements of $\tild G(\Q)$ are called coarse-conjugate if their quasi-semisimple parts are conjugate (in $G(\Q)$). Denoting the set of equivalence classes by $\O$, the geometric side will be an expansion 
		\[ J^{\tild G, T}(f) = \sum_{\o \in \O} J^{\tild G, T}_{\o \in \O} (f) \] 
		which we shall define and extend to the class $\CCC(\tild G(\A), K)$. 
		
		Following \cite{CLL}*{Lecture 1, 2, 9} we define the ``basic identity" for $T \in\aaa_0$ with $d_0(T) > d_0$ as
		\[ k^T_\geom(x) = k^{\tild G, T}_\geom (x) = \sum_{\tild P \supseteq \tild P_0} (-1)^{\apg} \sum_{\xi \in P(\Q) \bs G(\Q)} \hat\tau_{\tild P}(H_0(\xi x)-T) k_{\tild P}(\xi x) \]
		where 
		\[k_{\tild P}(x) = \int_{N_P(\Q) \bs N_P(\A)} \sum_{\delta \in \tild P(\Q)} \omega(x) f^1(x^{-1}\delta n x) \d n. \]
		
		We can decompose $k^T_{\geom}(x)$ according to coarse conjugacy classes: 
		\[ k^T_{\geom}(x) = \sum_{\o \in \O} k^T_{\o}(x) \]
		where 
		\begin{align} \label{def_ko}
			 k^T_{\o(x)} & =  \sum_{\tild P \supseteq \tild P_0} (-1)^{\apg} \sum_{\xi \in P(\Q) \bs G(\Q)} \hat\tau_{\tild P}(H_0(\xi x)-T) k_{\tild P, \o}(\xi x), \\
		 k_{\tild P, \o}(x) & = \int_{N_P(\Q) \bs N_P(\A)} \sum_{\delta \in \tild P(\Q) \cap \o} \omega(x) f^1(x^{-1}\delta n x) \d n. \nonumber
		 \end{align}
		
		The last equality follows from a basic fact \cite{MR518111}*{p. 923} that 
		\[ \tild P(\Q) \cap \o = (\tild M_P(\Q) \cap \o) N_P(\Q). \]
		We also set 
		\[ k_\geom(x) = \sum_{\gamma \in \tild G(\Q)} \omega(x) f^1(x^{-1} \gamma x);\quad  k_{\o}(x) = \sum_{\gamma \in \o} \omega(x) f^1(x^{-1} \gamma x). \]
		
		Following \cite{CLL}, Labesse and Waldspurger \cite{LW} show that in the expression 
		\[ J^T(f) = \sum_{\o \in \O} \int_{X_G} k^T_{\o}(x) \d x, \]
		only finitely many coarse conjugacy classes $\o \in \O$ give a nonzero contribution depending on the support of $f$ and the integral is absolutely convergent for all $T \in \aaa_0$ with $d_0(T)$ large enough. By the partition lemma of Langlands and Arthur (see \cref{partition}), we obtain
		\begin{equation} \label{k_o}
			k^T_{\o}(x) = \sum_{\substack{\tild P, Q : \\ P_0 \se Q \se P}} \sum_{\xi \in Q(\Q) \bs P(\Q)} (-1)^{\apg} F^Q_{P_0}(\xi x, T) \tau^P_Q(H_0(\xi x)-T) \hat\tau_{\tild P}(H_0(\xi x)-T) k_{\tild P, \o}(\xi x)
		\end{equation}
		By a combinatorial identity of Langlands \cite{LW}*{Lemme 2.11.5}, we have 
		\[ \sum_{R \supseteq P} \tild \sigma^R_P = \tau^P_Q \hat\tau_{\tild P}. \]
		So making a change of variables, we can write
		\[ k^T_{\o}(x) = \sum_{Q \se R} \sum_{\xi \in Q(\Q) \bs G(\Q)} F^Q_{P_0}(\xi x, T) \tild \sigma^R_Q(H_0(\xi x)-T) k_{\o, Q, R}(\xi x) \]
		where 
		\[ k_{\o, Q, R}(x) = \sum_{\substack{\tild P \supseteq \tild P_0:} \\ Q \se P \se R} (-1)^{\apg} k_{\tild P, \o}(x). \]
		
		The twisted version of Finis-Lapid's extension of the geometric side is as follows. 

		\begin{theorem} \label{geom}
			Assume the Root Cone lemma (\cref{RCL}) holds for the pair $(G, \theta)$. 
			\begin{enumerate}
				\item For any $f \in \CCC(\tild G(\A), K), \o \in \O$ and any $T \in \aaa_0$ suitably large multiple of the sum of positive coroots (see \cref{TCone}), the integrals 
				\[ J^T_{\geom}(f) = \int_{X_G} k^T_{\geom}(x) \d x \text{ and } \quad J^T_{\o}(f) = \int_{X_G} k^T_\O(x) \d x \]
				are absolutely convergent. 
				\item \label{part2} $J^T_{\geom}(f)$ and $J^T_{\o}(f)$ are polynomials in $T$ of degree $\leq a_{\tild P_0}-a_{\tild G}$ whose coefficients are continuous linear forms in $f$. 
				\item \label{part3} There is $r\geq 0$ and a continuous seminorm $\mu$ on $\CCC(\tild G(\A), K)$ such that 
				\begin{equation}\label{bound}
					\begin{split} \sum_{\o \in \O} \displaystyle\lmod{ \int_{X_G^T} k_{\o}(x) - J^T_{\o}(f)} \d x \leq \sum_{\o \in \O} \int_{X_G} \lmod{ F^G(x,T)k_{\o}(x) - k^T_{\o}(x) } \d x \\ \leq \mu(f) (1+\norm{T})^r \exp{(-d_0(T))}, \end{split}
				\end{equation}
				for any $f \in \CCC(\tild G(\A), K)$ and any $T \in \aaa_0$ ``suitably large'' .
				\item $J^T_{\geom}(f) = \sum_{\o \in \O} J^T_{\o}(f)$. 
			\end{enumerate}
			In addition, the upper bound on the seminorm $\mu$ and the coefficients in \cref{bound} of part (\ref{part3}) depends on the level of $K$ in the same way as in the non-twisted case \cite{FL16}*{Theorem 5.1}. 
		\end{theorem}
		
		We defer the proof of this theorem to \cref{maingeomproof}.
				
\subsection{A few technical results}

	In this section, we review some definitions and lemmas that will go into the proof of the geometric side. 
	
	\subsubsection{The modulus character}
	For $w \in W$ let $\delta_w$ denote the modulus function of $M_0(\A)$ on $N_w(\A) \bs N_0(\A)$ where $N_w = N_0 \cap w N_0 w^{-1}$. In particular, if $w = w_0$ is the long element in $W$, we denote $\delta_{w_0}$ by $\delta_0$. Denote also by $\delta_{M_0,N}$ the modulus function of $M_0(\A)$ on the unipotent radical $N$ of any parabolic subgroup of $G$. The following lemma is easy to prove, cf. \cite{Shahidi}*{\S 4.1}. 
	\begin{lemma}\label{modulus0}
		\begin{enumerate}
			\item $\delta_w = \displaystyle \frac{1}{2} \sum_{\substack{\alpha > 0, \\ w^{-1}  \alpha > 0}} \alpha, \ \delta_{M_0, N_w} = \frac{1}{2} \sum_{\substack{ \alpha > 0 \\ w^{-1} \alpha < 0} }\alpha.$
			\item $ \delta_0 = \delta_w . \delta_{M_0, N_w}.$
		\end{enumerate}
	\end{lemma}
		
	\begin{lemma}\label{modulus}
		\[ \delta_0(a^{-1} n_w b n_w^{-1}) \delta_{w^{-1}}(a^2) = \delta_{w^{-1}}(ab^{-1}) \delta_{M_0,N_{w^{-1}}}(a^{-1}b), \]
		for any representative $n_w$ of $w \in W$ and any $a, b \in A_0$. 
	\end{lemma}
		
	\begin{proof}
		Using the above \cref{modulus0}, we have 
		\begin{flalign*}
			\delta_0(a^{-1}) \delta_{w^{-1}}(a^2) & = \delta_{w^{-1}}(a^{-1}) \delta_{M_0,N_{w^{-1}}}(a^{-1}) \delta_{w^{-1}}(a^2) \\
				& = \delta_{w^{-1}}(a) \delta_{M_0,N_{w^{-1}}}(a^{-1})
		\end{flalign*}
		Following the proof of lemma 2.1 in \cite{FL11a} we can write
		\begin{flalign*}
			\delta_0(n_w b n_w^{-1}) & = \delta_{M_0,N_{w^{-1}}}(b) \delta_w(n_w b n_w^{-1}) \\					& = \delta_{M_0,N_{w^{-1}}}(b) \delta_{w^{-1}}(b^{-1}). 
		\end{flalign*}
		Multiplying the two gives the desired equality. 
	\end{proof}

	\subsubsection{Twisted Bruhat decomposition} \label{twistedBruhat}
		
		The Bruhat decomposition in the twisted case is similar to the usual case. Since the minimal parabolic $P_0$ is chosen to be $\theta_0$-stable, any element $\tild w = w \delta_0$ of the Weyl set $\tild W = W^{\tild G}$ gives a twisted Bruhat cell 
		\[ C(\tild w) = P_0 \tild w P_0 = P_0 (w \delta_0) P_0 = (P_0 w P_0) \delta_0;\]
		and $\tild G$ is the union of such cells $C(\tild w)$.
		 
		Note that subsets of $\Delta_{P_0}$ are not always in bijection with standard parabolic subsets but one needs to consider $\theta_0$-stable subsets of $\Delta_{P_0}$. For $Q \in \PPP$ there exist $Q^-, Q^+ \in \PPP$ such that $Q^- \se Q \se Q^+$ and $\Delta_0^{Q^-}, \Delta_0^{Q^+}$ are $\theta$-stable (see \cref{section:notation}). The twisted parabolic subgroups 
		\[ \tild Q^- := Q^- \delta, \quad \tild Q^+ := Q^+ \delta\] 
		are therefore well-defined for any choice $\delta \in \tild G(\Q)$ where $\theta = \Ad(\delta)$. Indeed, standard $\theta_0$-stable parabolic subsets are the `right' parabolic subsets one needs to consider to get the alternating sum in the kernel of the twisted trace formula. Following the notation of \cite{LW}*{p. 133}, for $Q \in \PPP$ define 
	\[ \tild G(Q, G) := \tild G(\Q) \bs \bigcup_{\tild Q^+ \se \tild P' \subsetneq \tild G} \tild P'(\Q). \]
		Being a bi-$Q(\Q)$-invariant set $\tild G(Q, G)$ is a finite disjoint union of twisted Bruhat cells $C(\tild w)$ over $\tild w \in \tild W(Q, G)$ or equivalently, over $w \in W$ satisfying $Q.Q(w) = G$, where $Q(w)$ is the smallest standard parabolic subgroup containing $w$. 
	
	\subsubsection{Mellin transform} \label{mellin}
	For a function $F \in \CCC^\infty_c(G(\A))$, we recall the definition of Mellin transform on $A_0$ and inversion formula: 
	\[	\phi(\lambda)(g) = \int_{A_0} F(ag) a^{\lambda + \rho_0} \d a \]
	The function can be recovered by the inverse-Mellin transform
	\[	F(ag) = \int_{\Re \lambda = \lambda_0} \phi(\lambda)(g) \delta_0(a)^{\frac{1}{2}} a^\lambda \d \lambda,\]
	where $\lambda_0 \in \aaa_0^*$ and for convenience, we have denoted $\exp{(\sprod{\lambda}{H_0(a)})}$ by $a^\lambda$. 	

	\subsubsection{Intertwining operators} \label{easyintertwining}
	
	We briefly recall the properties of intertwining operators; following \cite{FL16}, we will need it for principal series representations only. Later in the analysis of the spectral side we will define them more generally for any two associated parabolic subgroups. The space of representations parabolically induced from $P_0(\A)$ is defined by
	\[ I(\lambda) = \{ \varphi : G(\A) \to \C \text{ smooth } |\ \phi(pg) = \exp \sprod{\lambda + \rho_0}{H_0(p)}\phi(g) \; \forall p \in P_0(\A), g \in G(\A). \} \]
	The intertwining operator is a map 
	\[ M(w, \lambda) : I(\lambda) \to I(w \lambda) \]
	given by 
	\[ M(w, \lambda)\phi(g) = \int_{N_w(\A) \setminus N_0(\A)} \phi(n_w^{-1} n g) \d n. \]
	It is well-known that the integral over $\lambda$ is a product of local integrals and converges for $\lambda$ in the positive Weyl chamber ``sufficiently away from the origin". It extends meromorphically to $\aaa_{0, \C}^*$ with simple poles occurring on the root hyperplanes \cite{MW}*{IV.1}. Moreover following the notations in \cite{FL16}*{\S 3.3},
	\[ M(w, \lambda) \phi = m(w, \lambda) \phi \]
	where 
	\[ m(w, \lambda) = \prod_{\substack{\alpha \in \RRR_0 : \\ w^{-1} \alpha < 0} } m_\alpha(\sprod{\lambda}{\alpha^\vee}) \]
	and if $\lambda_0$ is in the positive Weyl chamber of $\aaa_0^*$ satisfying $\sprod{\lambda_0 - \rho_0}{\varpi^\vee} > 0$ for every $\varpi^\vee \in \hat\Delta_0^\vee$, then the function 
\begin{equation} \label{moderategrowth}
	\lambda \mapsto \prod_{\substack{ \alpha \in \Delta_0 :\\ w^{-1} \alpha < 0} } \sprod{\lambda}{\alpha^\vee} m(w, \lambda)
\end{equation}
	is holomorphic and of moderate growth on $\norm{\Re \lambda - \lambda_0} < \epsilon$ for some $\epsilon > 0$ sufficiently small. See \cite{MW}*{IV.1.11} and \cite{HC}*{Lemma 101} for details. 

	\subsection{Root Cone Lemma}	
	
	The Root Cone \cref{RCL} will be used to prove the finiteness of derivatives of $\phi_{T, Q, \l}(\lambda)$ in \cref{lemma46}. We will prove this for various pairs $(G, \theta_0)$ in \cref{section:RCL} including all cases when $G$ is split over $\Q$ (see \cref{rem:defQ}). Note that this lemma depends only on the semisimple part of $G$. The continuity of the geometric side for groups $G$ which are quasisplit but not split is conditional on proving this lemma which we assume to hold in this section. 
	
	Denote by $\Delta_0^{Q(w)}$ the subset of $\Delta_0$ corresponding to the smallest standard parabolic subgroup $Q(w)$ containing a representative of $w \in W$. Set
	\[ \gamma(w, \theta_0) := \displaystyle\frac{1}{2}(1 - \theta_0^{-1})
								\left(\sum_{\substack{\alpha: \\ \alpha, w\alpha > 0}} \alpha
								- \sum_{\substack{\beta : \beta > 0 \\ w\beta < 0}} \beta, \right) \]
	and for $\gamma \in \aaa_0^*$, 
	\[ \eta(\lambda, w, \theta_0, \gamma) := \lambda - \theta_0^{-1}w^{-1}\lambda - \gamma. \]
	Thus $\eta(\lambda, w, \theta_0, 0)$ denotes the vector $\lambda - \theta_0^{-1}w^{-1}\lambda$. Moreover, when $\gamma = \gamma(w, \theta_0)$ we will denote
	\begin{equation} \label{eta}
		\eta(\lambda, w, \theta_0) := \lambda - \theta_0^{-1}w^{-1}\lambda - \gamma(w, \theta_0).
	\end{equation}
	
	\begin{lemma} \label{sumcoeffs}
		\begin{enumerate}
			\item If $\lambda$ is any vector in $\aaa_0^*$ then 
				\[ \sum_{\varpi^\vee \in \hat\Delta_0^\vee} \sprod{(1-\theta_0^{-1}) \lambda}{\varpi^\vee} = 0.\]
			\item Suppose $w \in W, w \neq 1$ and $\lambda$ is in the (open) positive Weyl chamber of $\aaa_0^*$ then 
			 	\[ \sum_{\beta^\vee \in \Delta_0^\vee} \sprod{\lambda - \theta_0^{-1} w^{-1} \lambda}{\varpi_\beta^\vee} > 0.\]
		\end{enumerate}
	\end{lemma}

	\begin{proof}
		The first statement follows because $\theta_0^{-1}$ is a permutation on the set $\Delta_0$ or equivalently on $\hat\Delta_0^\vee$. The two parts of the lemma estimate the sum of coefficients of respective vectors expressed in the basis $\{ \beta \in \Delta_0\}$ of roots. We write $\lambda - \theta_0^{-1} w^{-1} \lambda$ as a sum of $\lambda - w^{-1} \lambda$ and $(1-\theta_0^{-1})(w^{-1} \lambda)$. If $w \neq 1$ then $\Delta_0^{Q(w)}$ is nonempty so by~\cite{MR1890629}*{Ch. VI \S 1.6 Proposition 18} and the choice of $\lambda$, the inner product $\sum_{\beta^\vee \in \Delta_0^\vee} \sprod{\lambda - w^{-1} \lambda}{\varpi_\beta^\vee}$ is positive. The other inner product sum $\sum_{\beta^\vee \in \Delta_0^\vee} \sprod{(1-\theta_0^{-1})(w^{-1} \lambda)}{\varpi_\beta^\vee}$ vanishes using Part 1. 
	\end{proof}

	\begin{lemma}[Root Cone lemma] \label{RCL}
		For $w \in W, w \neq 1$, there exists an open cone $\Omega_0$ inside the positive Weyl chamber $(\aaa_0^*)^+$ in $\aaa_0^*$ such that for every $\lambda \in \Omega_0$ and every $\beta \in \Delta_0^{Q(w)}$, 
		
		\begin{equation} \label{RCLinequality}
			\sprod{\lambda - \theta_0^{-1}w^{-1}\lambda}{\varpi_\beta^\vee} > 0. 
		\end{equation}
	\end{lemma}
		
	\begin{remark} \label{modifiedRCL}
		The version of Root Cone Lemma we will use requires the estimate of the vector $\eta(\lambda, w, \theta_0, \gamma)$ rather than $\eta(\lambda, w, \theta_0, 0)$, i.e., $\gamma \neq 0$. By choosing $\lambda \in (\aaa_0^*)^+$ suitably away from the origin we can ensure for fixed $\gamma \in \aaa_0^*$ that 
		\[ \sprod{\eta(\lambda, w, \theta_0, \gamma)}{\varpi_\beta^\vee} > 0 \]
	for all $\beta \in \Delta_0^{Q(w)}$. Throughout this section fix $\gamma \in \aaa_0^*$ and the open subset $\Omega_\gamma$ of points $\lambda$ satisfying this condition. It is clear that the $\Omega_0$ in \cref{RCL} corresponds to $\Omega_\gamma$ when $\gamma=0$. We need the RCL to get the two estimates below. 
	\end{remark}
	
	\begin{lemma} \label{psiTQLconv}
		Assume $\gamma \in \aaa_0^*, Q\in \PPP$ and $w \in \tilde W(Q, G)$. Then for every $\lambda \in \aaa_0^*$ with $\Re(\lambda) \in \Omega_\gamma$, the integral 
		\[ \psi_{T, Q, l}(\lambda) := \int_{\aaa_Q} \exp \sprod{X_Q}{-\eta(\lambda, w, \theta_0, \gamma)} \tau_Q(X_Q-T) \sum_{\alpha \in \Delta_Q} \sprod{\alpha}{X_Q-T}^l \d X_Q \]
		converges absolutely. 
	\end{lemma}
	
	\begin{proof}
		By the condition on $w$, we have $\Delta_0^Q \cup \Delta_0^{Q(w)} = \Delta_0$. The Root Cone Lemma ensures that
		\[ \sprod{\varpi^\vee}{\eta(\lambda, w, \theta_0, \gamma)} > 0 \]
		 whenever $\lambda \in \Omega_\gamma$ and $\varpi^\vee \in (\hat\Delta_0^{Q(w)})^\vee$. Since elements of $\Delta_Q$ are restrictions of elements in $\Delta_0~\setminus~\Delta_0^Q \subseteq \Delta_0^{Q(w)}$ to $\aaa_Q$, the inner product $\sprod{X_Q}{\eta(\lambda, w, \theta_0, \gamma)}$ is negative. The decaying exponential term in the above integral thus dominates the polynomial term, giving the required absolute convergence. 
	\end{proof}
	
	\begin{lemma} \label{TCone}
		There is an unbounded subset of the line $\R(\sum_{\hat\varpi^\vee \in \hat\Delta_0^\vee} \hat\varpi^\vee)$ in $\aaa_0$ independent of $w \in W$ such that given $\gamma \in \aaa_0^*$, we have
		\[ \sprod{\eta(\lambda, w, \theta_0, \gamma)}{T} > 0 \]
		for every $\lambda \in \Omega_\gamma$ whenever $T$ belongs to this set. 
	\end{lemma}
	
	\begin{proof}
		Up to a positive number, the above inner product is the sum of coordinates of $\eta(\lambda, w, \theta_0, \gamma)$ in the basis $\Delta_0$ of roots and the estimate follows by applying Part 1 (respectively Part 2) of \cref{sumcoeffs} to the vector $\gamma(w, \theta_0)$ (resp. $\eta(\lambda, w, \theta_0, 0)$). The independence on $w$ is also from Part 2. 
	\end{proof}
	
	\subsection{Two theorems}
	We now state two theorems which will give crucial estimates towards proving the main result on the geometric side and they will be proven in \cref{section:3.2tild}. 
	
\begin{theorem} \label{3.1tild}
	There exists an integer $r \geq 0$, a vector $\xi(Q) \in \aaa_0^*$ with $\sprod{\xi(Q)}{\beta} > 0$ for every $\beta \in \Delta_Q$ and a seminorm $\mu$ on $\CCC(\tild G(\A), K)$ such that for every $Q\in \PPP$ and $l \geq 0$, 
	\begin{multline} \label{isBounded}
		\int_{Y_Q} F^Q(x,T) \tau_Q(H_Q(x)-T) \norm{H_Q(x)-T_Q}^l \sum_{\gamma \in \tild G(Q, G)} \mod{f^1(x^{-1}\gamma x)} \d x \\ \ll (1+\norm{T})^r \exp{-\sprod{\xi(Q)}{T}} \mu(f),
	\end{multline}
	holds for every $f \in \CCC(\tild G(\A), K)$ and $T \in \aaa_0$ suitably large multiple of the sum of positive coroots (see \cref{TCone}). Moreover, the seminorm $\mu$ satisfies the same bound as in the non-twisted case (\cite{FL16}*{Theorem 5.1}). 
\end{theorem}
	
\begin{remark} \label{k_oestimate}
	In the above sum, recall from \cref{twistedBruhat} that  
	\[ \tild G(Q, G) := \tild G(\Q) \bs \bigcup_{\tild Q^+ \se \tild P' \subsetneq \tild G} \tild P'(\Q). \]
	If $Q=G$ then the inequality reduces to
	\[ \int_{X_G} F^G(x, T) \sum_{\gamma \in \tild G(\Q)} \mod{f^1(x^{-1}\gamma x)} \d x \ \leq \ \mu(f) (1+\norm{T})^r \exp{(-d_0(T))} \]
	and the LHS is just $\int_{X_G^T} \sum_{\o \in \O} \mod{k_{\o}(x)} \d x$. 
\end{remark}

\begin{theorem} \label{3.2tild} 
	Let $\tild Q$ be a standard parabolic subset and $\tild w = w \delta_0 \in \tild W(Q, G)$. Then there is an integer $r \geq 0$, a vector $\xi(Q) \in \aaa_0^*$ with $\sprod{\xi(Q)}{\beta} > 0$ for all $\beta \in \Delta_Q$ and a seminorm $\mu$ on $\CCC(\tild G(\A), K)$ such that
	\begin{multline*}
		\int_{N_w(\A) \bs N_0(\A)} \int_{A_0} \int_{N_0(\A)} \int_{M_0(\A)^1} \mod{f^1(a^{-1} n^{-1} n_{\tild w} a u m )} \chi(a) \d m \d u \d a \d n \\ \ll_{K, l} \mu(f) (1+\norm{T})^r \exp{-\sprod{\xi(Q)}{T}},
	\end{multline*}
	holds for every $l \geq 0, f \in \CCC(\tild G(\A), K)$ and $T \in \aaa_0$ a suitably large multiple of the sum of positive coroots (see \cref{TCone}). Here, $N_w = N_0 \cap n_w N_0 n_w^{-1}$ and 
	\[	\chi(a) = \chi_{T, Q, l}(a) = \tau_Q(H_Q(a) - T)\ \hat \tau_{P_0}^Q(T - H_0(a)) \tau_{P_0}^Q(H_0(a)-T_1) \sum_{\alpha\in \Delta_Q} \sprod{H_Q(a)-T_Q}{\alpha}^l. 	\]
\end{theorem}

\subsection{Continuity of the geometric side} \label{maingeomproof}

	In this section, we prove \cref{geom} that the distribution $J^{\tild G}_{\text{geom}}(f)$ initially defined for $f \in \CCC_c^\infty(\tild G(\A))$ extends to $\CCC(\tild G(\A), K)$. As in the non-twisted case, the proof involves modifying Arthur's (or rather, Labesse-Waldspurger's) proof in the compactly supported setting to our case. We imitate the method of Finis-Lapid \cite{FL16} whenever possible. 
		
	\begin{proof} {[of \cref{geom}]}
		Since $f$ is not of compact support, we first prove why 
		\[ k_{\tild P}(x) = \int_{N_P(\A)} \sum_{\gamma \in \tild M_P(\Q)} \omega(x) f^1(x^{-1} \gamma n x) \d n \]
		and 
		\[k^T_{\geom}(x) = \sum_{\o \in \O} k_{\o}^T(x) \]
		(see \cref{def_ko}) are well-defined. The corresponding expressions in the non-twisted case follow directly from the assumption on $f$ and  \cite{FL16}*{Lemma 2.1 (1)} which estimates a discrete sum over the $\Q$-points of $M_P$ by an integral over a finite sum of derivatives of the test function. To prove $k_{\tild P}(x)$ is well-defined we additionally use \cref{bijectionlemma}. The finiteness of $k_{\geom}^T(x)$ is now  a standard application of \cite{MR518111}*{Lemma 6.4}. 
	
		For $\tild P \supseteq \tild P_0$ and $f \in \CCC(\tild G(\A), K)$ we could replace the sum over $\tild M_P(\Q)$ in the definition of 
		\[ k_{\tild P}(x) = \int_{N_P(\A)} \sum_{\gamma \in \tild M_P(\Q)} \omega(x) f^1(x^{-1} \gamma n x) \d n \]
		by an integral over $\tild M_P(\A)$ of a finite sum of derivatives of $f$, following \cref{bijectionlemma} and \cite{FL16}*{Lemma 2.1(1)}. Thus each $k_{\tild P}(x)$ hence $k^T_{\geom}(x)$ is well-defined. We can formally write 
		\[ \sum_{\o \in \O} \lmod{ J^T_{\o}(f)} \leq \sum_{\o \in\O} \lmod{ J^T_{\o}(f) - \int_{X_G} k^T_{\o}(x) \d x } + \sum_{\o \in \O} \int_{X_G} \mod{ k^T_{\o}(x) } \d x. \]
		By \cref{k_oestimate}, it follows that to prove $J^T_\geom(f)$ and $J^T_{\o}(f)$ exist and the relation $J^T_\geom(f) = \sum_{\o \in \O} J^T_{\o}(f)$, it suffices to prove part (\ref{part3}). Part (\ref{part2}) is a formal property which holds whenever  $J^T_{\o}(f)$ is absolutely convergent, cf. \cite{LW}*{Th\'{e}or\`{e}me 11.1.1}. We now prove part (\ref{part3}). 
			
		The first inequality is obvious. Recall from \cref{k_o} that
		\[ k^T_{\o}(x) = \sum_{Q \se R} \sum_{\xi \in Q(\Q) \bs P(\Q)} F^Q_{P_0}(\xi x, T) \tild \sigma^R_Q(H_0(\xi x)-T) k_{\o, Q, R}(\xi x). \]
		Using twisted Levi decomposition we can write 
		\[ k_{\o, Q, R}(x) = \sum_{\substack{S: \\ Q \se S \se R}} k_{\o, Q, R, S}(x)\]
		where 
		\[ k_{\o, Q, R, S}(x) = \sum_{\eta \in \tild M_S(Q, S)} \sum_{\substack{\tild P : \\ \tild S \se \tild P \se \tild R^-}} (-1)^{\apg} \sum_{\nu \in N^S_P(\Q)} \int_{N_P(\A)} \omega(x) f^1(x^{-1}\eta \nu nx) \d n, \]
		\[ \tild M_S(Q, S) = \tild M_S(\Q) \bs \bigcup_{\substack{\tild P' : \\ \tild Q^+ \se \tild P' \subsetneq \tild S}} \tild P'(\Q). \]
									
		Here we are using that if $\tild P \supseteq \tild P_0$ is such that $Q \se P \se R$ then $\tild Q^+ \se \tild P \se \tild R^-$. Note that if $Q=R$ then $\tild \sigma^R_Q$ vanishes unless $Q = R = G$ in which case it is 1. Hence the contribution from $Q=R$ is precisely 
		\[ F^G(x,T) k_{\o, G, G}(x) = F^G_{P_0}(x,T) k_{\o}(x). \]
		Thus, 
		\[ F^G_{P_0}(x,T)k_{\o}(x) - k^T_{\o}(x) = \sum_{Q \subsetneq R} \sum_{\xi \in Q(\Q) \bs G(\Q)} F^Q_{P_0}(\xi x, T) \tild \sigma^R_Q(H_0(\xi x)-T) k_{\o, Q, R}(\xi x). \]
		Making a change of variables, 
		\[ \sum_{\o \in \O} \int_{X_G} \lmod{ F(x, T) k_{\o}(x) - k^T_{\o}(x) } \d x = \sum_{\o \in \O} \sum_{Q \subsetneq R} \int_{Y_Q} F^Q_{P_0}(x, T) \tild \sigma^R_Q(H_0(x)-T) \mod{k_{\o, Q, R}(x)} \d x. \]
		Fixing $Q \se S \se R$ with $Q \neq R$, we need to estimate
		\[ \sum_{\o \in \O} \int_{Y_Q} F^Q(x,T) \tild \sigma^R_Q(H_0(x)-T) \mod{k_{\o, Q, R, S}(x)} \d x . \]
		We can assume that $\tild \sigma^R_Q(H_Q(x)-T)$ is 1 and so is $\tau^R_Q(H_Q(x)-T)$. Now invoke \cite{FL16}*{Corollary 4.7}, the hypotheses that $F^Q(x,T) \tau^R_Q(H_Q(x)-T)=1$ being satisfied. (Since the dependence of $\mu$ on the level of $K$ in the regular (non-twisted) case is only via this Corollary whose twisted equivalent remains same, the corresponding bound in the twisted case remains the same.) The aforementioned Corollary tracks Arthur's proof in \cite{MR518111,MR625344} by applying Iwasawa decomposition to $x$ which is justified since the expression above is left invariant under $Q(\Q)N_R(\A)$. Up to a constant whose dependence on $K$ is tracked in \cite{FL16}*{Corollary 4.7},
		\[ \lmod{k_{\o, Q, R, S}(x) } \ll_K \exp{ -\left( \sprod{\xi^R_S}{T}+\sprod{(\xi^R_S)_Q}{H_0(x)-T} \right) } \sum_{\eta \in \tild M_S(Q, S)} \int_{N_S(\A)} \mod{f^1(x^{-1}\eta n x)} \d n \d x.\]
			
		The remaining steps to reduce this to \cref{3.1tild} are the same as in \cite{FL16}*{p. 21} with their twisted analogues, namely
\begin{itemize}
	\item by using \cref{bijectionlemma}, we can invoke \cite{FL11a}*{Lemma 3.4} to assume $f\geq 0$ and $K = \K_f$; 
	\item we can apply Iwasawa decomposition with respect to $S$ and use lemma 4.8 of \cite{FL16} verbatim. 
\end{itemize} 
	The bound now follows by \cref{3.1tild}. 
			
	\end{proof}


\section{Proofs of Theorems \ref{3.1tild} and \ref{3.2tild}} \label{section:3.2tild}
	We will apply the twisted Bruhat decomposition to \cref{3.1tild} to reduce it to \cref{3.2tild}. The proof of \cref{3.2tild} involves estimating the integrals over twisted Bruhat cells. We will apply the Mellin transform and use the slow growth of intertwining operators with the estimate of \cite{FL16}*{Proposition 3.4} to get the required estimate. 
		
	\subsection{Reduction of \cref{3.1tild} to \cref{3.2tild}}
	
	The estimate below for $Q \in \PPP$ and any left $Q(\Q)-$invariant measurable function $f$ on $G(\A)^1$, 
	\[ \int_{Y_Q} \mod{f(x)} \d x \leq \int_\K \int_{N_0(\Q)\bs N_0(\A)} \int_{A_0} \int_{X_M} \mod{f(namk)} \delta_0(a)^{-1} \tau^Q_{P_0}(H_0(a) - T_1) \d m \d a \d u \d k \]
	which occurs in \cite{FL16}*{Equation 2} remains true in the twisted case. Indeed, averaging over $\tild G(Q, G)$ makes the integrand on the left hand side of \cref{isBounded} bi-$Q(\Q)$-invariant. Applying this estimate we need to bound 
	\[ \int_{\K} \int_{N_0(\Q)\bs N_0(\A)} \int_{A_0} \int_{X_M}  \sum_{\gamma \in \tild G(Q, G)} \mod{f^1\left( (namk)^{-1} \gamma (namk) \right)} \; \chi(a) \delta_0(a)^{-1} \d m \d a \d u \d k. \]
	Here, 
	\[ \chi(a) = F^Q(a, T)\ \tau_Q(H_Q(a) - T)\ \tau_{P_0}^Q(H_0(a)-T) \sum_{\alpha \in \Delta_Q} \sprod{H_Q(a)-T_Q}{\alpha}^l. \]
	Ignoring the integration on the compact sets $\K$ and $X_{M_0}$ by \cite{FL16}*{Prop. 2.1(2)} we are reduced to bounding the integral 
	\[ \int_{N_0(\Q)\bs N_0(\A)} \int_{A_0} \sum_{\gamma \in \tild G(Q, G)} \mod{f^1\left( (na)^{-1} \gamma na \right)} \chi(a) \delta_0(a)^{-1} \d a \d n. \]
		
	Applying the twisted Bruhat decomposition to $\tild G(Q, G)$ by \cref{twistedBruhat}, we need to estimate for fixed $\tild w = w \delta_0 \in \tild G(Q, G)$, 
		\[ \int_{N_0(\Q) \bs N_0(\A)} \int_{A_0} \sum_{u_2 \in N_w(\Q)\bs N_0(\Q)} \sum_{u_1 \in N_0(\Q)} \sum_{m \in M_0(\Q)} \mod{f^1(a^{-1} n^{-1} u_2^{-1} m n_{\tild w} u_1 n a)} \ \chi(a) \delta_0(a)^{-1} \d a \d n. \]
		
		Applying \cite{FL16}*{Lemma 3.3} to the translate $g(u_1) = f^1(bu_1)$ with $b = n^{-1} u_2^{-1} m n_{\tild w}$ allows us to replace the sum over $u_1 \in N_0(\Q)$ by the integral of functions $g*X$ for $X$ ranging over a finite set of differential operators. Replacing $g$ by one such derivative (and that with $f^1$) reduces to bounding the sum-integral 
		\[ \int_{N_0(\Q) \bs N_0(\A)} \int_{A_0} \sum_{u_2 \in N_w(\Q)\bs N_0(\Q)} \int_{u_1 \in N_0(\A)} \sum_{m \in M_0(\Q)} \mod{f^1(a^{-1} n^{-1} u_2^{-1} m n_{\tild w} u_1 a)} \ \chi(a) \delta_0(a)^{-1} \d a \d n \]
		\[ = \int_{N_0(\Q) \bs N_0(\A)} \int_{A_0} \sum_{u_2 \in N_w(\Q)\bs N_0(\Q)} \int_{u_1 \in N_0(\A)} \sum_{m \in M_0(\Q)} \mod{f^1(a^{-1} n^{-1} u_2^{-1} m n_{\tild w} a u_1)} \ \chi(a) \d a \d n. \]
		(Combining the sum over $u_2$ and the integral over $n$ gives), 
		\[ \int_{N_w(\Q) \bs N_0(\A)} \int_{A_0} \int_{u_1 \in N_0(\A)} \sum_{m \in M_0(\Q)} \mod{f^1(a^{-1} n^{-1} m n_{\tild w} a u_1)} \ \chi(a) \d a \d n. \]
		Note that as a function of $n$, the inner integral is left $N_w(\A)$-invariant. Hence we can write 
		\[ \int_{N_w({\A}) \bs N_0(\A)} \int_{A_0} \int_{u_1 \in N_0(\A)} \sum_{m \in M_0(\Q)} \mod{f^1(a^{-1} n^{-1} m n_{\tild w} a u_1)} \ \chi(a) \d a \d n. \]
		
	Also, 
		\begin{multline*}
			a^{-1} u^{-1} m n_w a u_1 = a^{-1} u^{-1} (n_w m') a u_1 = a^{-1} u^{-1} n_w a m' u_1 \\ = a^{-1} u^{-1} n_w a (m' u_1 m'^{-1})m' = a^{-1} u^{-1} n_w a u_2 m'. 
		\end{multline*}
		Thus we need to bound 
		\[ \int_{N_w(\A) \bs N_0(\A)} \int_{A_0} \int_{u_1 \in N_0(\A)} \sum_{m \in M_0(\Q)} \mod{f^1(a^{-1} n^{-1} n_{\tild w} a u_1 m) } \ \chi(a) \d a \d n, \]
		which by \cite{FL16}*{Lemma 2.1(1)} and \cref{bijectionlemma} reduces to a finite sum of integrals over $M_0(\A)$ of derivatives of $f^1$. Replacing $f$ by one such derivative as before, we need to bound
		\[ \int_{N_w(\A) \bs N_0(\A)} \int_{A_0} \int_{u_1 \in N_0(\A)} \int_{M_0(\A)^1} \mod{f^1(a^{-1} n^{-1} n_{\tild w} a u_1 m) } \ \chi(a) \d m \d a \d n, \]
		which reduces to \cref{3.2tild}. \hfill$\clubsuit$
		
	For a finite dimensional space $V$, let $\DDD(V)$ denote the space of invariant diferential operators on $V$ with the standard filtration. The lemma below is a modification of lemma 3.5 of \cite{FL16} in the twisted setting. 
\begin{lemma} \label{lemma46}
	Suppose $Q$ is a standard parabolic subgroup of $G$ and $\tilde w = w \delta_0 \in \tilde W(Q, G)$. For $\gamma = \gamma(w, \theta_0)$ consider the set $\Omega = \Omega_\gamma$ in \cref{modifiedRCL}. Then the integral 
	\[	\varphi_{T, Q, l}(\lambda) := \int_{X \in \aaa_0} \exp \sprod{X}{-\eta(\lambda, w, \theta_0) } 
			\tau_Q(X - T)\ \hat \tau_{P_0}^Q(T - X) \tau_{P_0}^Q(X-T_1) 
			\sum_{\alpha\in \Delta_Q}  \sprod{X_Q-T_Q}{\alpha}^l \d X 
	\]
 	is absolutely convergent for $\Re(\lambda)$ in compact subsets of $\ \Omega$ and $T$ in the unbounded set of \cref{TCone}. Moreover, for fixed $\Lambda_0 \in \Omega$ and any differential operator $D \in \DDD(\aaa_0^*)$ of degree $d$, there is a vector $\xi(Q)$ such that $\beta(\xi(Q)) > 0$ for all $\beta \in \Delta_Q$ and 
	\[ \mod{(\varphi_{T, Q, l}*D)(\lambda)} \ \ll_{D, l} (1 + \norm{T})^{d + a_0^Q} \exp -\sprod{T}{\xi(Q)}\]
	where $\Re(\lambda) = \Lambda_0$.
\end{lemma}

\begin{proof}
	Similar to \cite{FL16}*{Lemma 3.5}, we can use the decomposition $\aaa_0 = \aaa_Q \oplus \aaa_0^Q$ to write $X = X_Q+ X^Q$ and 
	\[ \varphi_{T, Q, l}(\lambda) = \psi_T^Q(\lambda) . \psi_{T, Q, l}(\lambda) \]
	where
	\[ \psi_T^Q(\lambda) := \int_{\aaa_0^Q} \exp \sprod{X^Q}{-\eta(\lambda, w, \theta_0)}\tau_0^Q(X^Q-T_1) \hat{\tau}_0^Q(T-X^Q) \d X^Q,\]
	and 
	\[ \psi_{T, Q, l}(\lambda) := \int_{\aaa_Q} \exp \sprod{X_Q}{-\eta(\lambda, w, \theta_0)} \tau_Q(X_Q-T) \sum_{\alpha \in \Delta_Q} \sprod{\alpha}{X_Q-T}^l \d X_Q. \]
	The function $X^Q \mapsto \tau_0^Q(X^Q-T_1) \hat{\tau}_0^Q(T-X^Q)$ is the characteristic function of the convex hull of the set $\{ T_S^Q + T_1^S : P_0 \subseteq S \subseteq Q \}$ \cite{MR625344}*{Section 6}, hence the integral $\psi_T^Q(\lambda)$ is its Fourier transform evaluated at $-\eta(\lambda, w, \theta_0)$ which is also the smooth function corresponding to the aforementioned $(M_Q, M_0)$-orthogonal set. In particular it is compactly supported and using \cref{psiTQLconv}, we have the convergence of $\varphi_{T, Q, l}(\lambda)$. It remains to estimate the derivatives. Assume $D = D^Q D_Q$ where $D^Q \in \DDD((\aaa_0^Q)^*)$ and $D_Q \in \DDD(\aaa_Q^*)$ and let $d = d^Q + d_Q$ be their degrees. By loc.~cit. $\psi_T^Q(\lambda)$ equals
	\begin{flalign*} 
		& \sum_{S : P_0 \subseteq S \subseteq Q} \exp \sprod{T_S^Q + T_1^S}{-\eta(\lambda, w, \theta_0)} \epsilon_Q(-\eta(\lambda, w, \theta_0)) \\ 
		& = \sum_{S : P_0 \subseteq S \subseteq Q} \Vol((\Delta_S^Q)^\vee) \frac{\exp \sprod{T_S^Q + T_1^S}{-\eta(\lambda, w, \theta_0)}}{\prod_{\beta \in \Delta_S^Q} \sprod{-\eta(\lambda, w, \theta_0)}{\beta^\vee}}. 
	\end{flalign*}  
	Tracking the dependence on $T\in \aaa_0$, we have
	\[ 
		\mod{\psi_T^Q*D^Q(\lambda)} \ll_{D^Q} (1 + \norm{T})^{d^Q + a_0^Q}\sum_{P_0 \subseteq S \subseteq Q} \exp \sprod{T_S^Q}{-\eta(\Lambda_0, w, \theta_0)}.
	\]
	Observe by the conditions on $\Lambda_0$ and $T$, the inner product $\sprod{T_S^Q}{-\eta(\Lambda_0, w, \theta_0)}$ is negative and equals $\sprod{T_S^Q}{-\Lambda_0 + w^{-1} \Lambda_0}$ (using \cref{sumcoeffs}). Since $\Lambda_0 - w^{-1} \Lambda_0$ is a \underline{positive} linear combination of the roots in $\Delta_Q$, there exists a $\xi(Q) \in \aaa_0^*$ whose coefficients in the basis $\Delta_0$ are nonnegative (depending on $\Lambda_0$) such that $\sprod{T_S^Q}{-\Lambda_0 + w^{-1} \Lambda_0} \leq -\sprod{T^Q}{\xi(Q)}$. (By the choice of $T$), it is also possible to choose $\xi(Q)$ independent of $w$, for example we can choose $\xi(Q) = N \sum_{\beta \in \Delta_Q} \beta$ for $N$ suitably large. Therefore,
	\[ 
		\mod{\psi_T^Q*D^Q(\lambda)} \ll_{D^Q} (1 + \norm{T})^{d^Q + a_0^Q} \exp - \sprod{T^Q}{\xi(Q)}.
	\]
	The estimate for $\psi_{T, Q, \l}*D_Q(\lambda)$ is similar to that in lemma 3.5 of \cite{FL16} and we have,
	\[ \mod{(\psi_{T, Q, \l}*D_Q)(\lambda)} \ll_{D_Q, \l} (1 + \norm{T})^{d_Q} \exp -\sprod{T_Q}{\xi(Q)}. \]
\end{proof}

	\subsection{Proof of \cref{3.2tild}}
	
\begin{proof}
	The quantity to estimate is
	\[ \int_{N_w(\A) \bs N_0(\A)} \int_{A_0} \int_{N_0(\A)} \int_{M_0(\A)^1} \mod{f^1(a^{-1} n^{-1} n_{\tild w} a u m )} \chi(a) \d m \d u \d a \d n, \]
	where 
	\[ \chi(a) = \chi_{T, Q, l}(a) = \tau_Q(H_Q(a) - T)\ \hat \tau_{P_0}^Q(T - H_0(a)) \tau_{P_0}^Q(H_0(a)-T_1) \sum_{\alpha\in \Delta_Q} \sprod{H_Q(a)-T_Q}{\alpha}^l. \]
	
	We split the integral over $u \in N_0(\A)$ as $n' u$ where $n' \in N_{w^{-1}}(\A)$ and $u \in N_{w^{-1}}(\A)\bs~N_0(\A)$. Thus, we want to bound 
	\[ \int_{n \in N_w(\A) \bs N_0(\A)} \int_{A_0} \int_{n' \in N_{w^{-1}}(\A)} \int_{u \in N_{w^{-1}}(\A)\bs N_0(\A)} \int_{M_0(\A)^1} \mod{f^1(a^{-1} n^{-1} n_{\tild w} a n'u m )} \chi(a) \d m \d u \d n' \d a \d n.  \]
	We can conjugate $n'$ over $n_{\tild w} a$: 
	\begin{flalign*}
		n_{\tild w} an' & = n_w \delta_0 a n'\\
			&= n_w \delta_0 (an' a^{-1}) \delta_0^{-1} n_w^{-1} n_w \delta_0 a \\
			&= n_w (\delta_0 n'' \delta_0^{-1}) n_w^{-1} n_w \delta_0 a \quad \cdots n''= an'a^{-1} \in N_{w^{-1}}(\A) \\
			&= n_w \theta_0(n'') n_w^{-1} n_{\tild w} a \\
				&= n_w n''' n_w^{-1} n_{\tild w} a
	\end{flalign*}
	where $n''' \in N_{w^{-1}}(\A) = N_0(\A) \cap w^{-1}N_0(\A)w$. Therefore $n_1 := n_w n''' n_w^{-1} \in N_w(\A)$. Making this change of variable, we are reduced to bounding 
	\begin{multline*} \int_{n \in N_w(\A)\bs N_0(\A)} \int_{n_1 \in N_w(\A)} \int_{A_0} \int_{M_0(\A)^1} \int_{u \in N_{w^{-1}}(\A)\bs N_0(\A)} \mod{f^1(a^{-1}n^{-1}n_1 n_{\tild w} a u m)}\chi(a). \\ \delta_{M_0,N_{w^{-1}}}(a)^{-1} \d u \d m \d a \d n_1 \d n 
	\end{multline*}
	\begin{flalign*}
		& = \int_{N_0(\A)} \int_{A_0} \int_{M_0(\A)^1} \int_{N_{w^{-1}}(\A)\bs N_0(\A)} \mod{f^1(a^{-1}n n_{\tild w} a u m)}\chi(a) \delta_{M_0,N_{w^{-1}}}(a)^{-1} \d u \d m \d a \d n \\
		& = \int_{N_0(\A)} \int_{A_0} \int_{M_0(\A)^1} \int_{N_{w^{-1}}(\A)\bs N_0(\A)} \mod{f^1(a^{-1}n n_w \theta_0( a u m) \delta_0)}\chi(a) \delta_{M_0,N_{w^{-1}}}(a)^{-1} \d u \d m \d a \d n \\
		& = \int_{N_0(\A)} \int_{A_0} \int_{M_0(\A)^1} \int_{N_{w^{-1}}(\A)\bs N_0(\A)} \mod{h(a^{-1}n n_w \theta_0( a u m))}\chi(a) \delta_{M_0,N_{w^{-1}}}(a)^{-1} \d u \d m \d a \d n, \\
		& = \int_{N_0(\A)} \int_{A_0} \int_{M_0(\A)^1} \int_{N_{w^{-1}}(\A)\bs N_0(\A)} \mod{h(a^{-1}n \theta_0(m) n_w \theta_0( a u))}\chi(a) \delta_{M_0,N_{w^{-1}}}(a)^{-1} \d u \d m \d a \d n \\
		& = \int_{N_0(\A)} \int_{A_0} \int_{M_0(\A)^1} \int_{N_{w^{-1}}(\A)\bs N_0(\A)} \mod{h(n \theta_0(m) a^{-1} n_w \theta_0( a u))}\chi(a) \delta_{M_0,N_{w^{-1}}}(a)^{-1} \delta_0(a) \d u \d m \d a \d n \\
		& = \int_{P_0(\A)^1} \int_{A_0} \int_{N_{w^{-1}}(\A)\bs N_0(\A)} \mod{h(p a^{-1} n_w \theta_0( a u))}\chi(a) \delta_{M_0,N_{w^{-1}}}(a)^{-1} \delta_0(a) \d u \d a \d p.
	\end{flalign*}
			
	In the above chain of equalities we repeatedly use the facts that $M_0(\A)^1$ is invariant under $\theta_0$ and that $N_0(\A)$ normalizes $M_0(\A)^1$. Moreover, since $f^1 \in \CCC(\tild G(\A), K)$ so by \cref{bijectionlemma}, the function $h \in \CCC(G(\A), K)$ where $h(x) = f^1(x\delta_0)$. By an application of \cref{modulus0}, this equals 
	\[ \int_{P_0(\A)^1} \int_{A_0} \int_{N_{w^{-1}}(\A)\bs N_0(\A)} \mod{h(p a^{-1} n_w \theta_0( a u))}\chi(a) \delta_{w^{-1}}(a) \d u \d a \d p. \]
				
	Recall the definition of the space of principal series representations in \cite{FL11a}*{\S 3.3} and in particular, that of $F_h$ for $h \in\CCC^\infty_c(G(\A))$, 
	\[ F_h(g) = \int_{P_0(\A)^1} h(pg) \d p. \]
			
	Thus we want to consider
	\[ \int_{A_0} \int_{N_{w^{-1}}(\A)\bs N_0(\A)} F_{\mod{h}}(a^{-1} n_w \theta_0( a u))\chi(a) \delta_{w^{-1}}(a) \d u \d a. \]

	As $u$ is integrated over $N_{w^{-1}}(\A) \bs N_0(\A)$, so is $\theta_0(u)$ by a change of variables. (Because $\theta_0: W^G \to W^G$ maps $n_w$ to another representative of $w \in W$.) So look at 
	\[ \int_{A_0} \int_{N_{w^{-1}}(\A)\bs N_0(\A)} F_{\mod{h}}(a^{-1} n_w \theta_0(a) u)\chi(a) \delta_{w^{-1}}(a) \d u \d a. \]
			
	By \cite{FL11a}*{Lemma 3.4} and \cref{bijectionlemma}, we can assume that $f$, hence $F_h$ is of compact support. This justifies taking the Mellin transform below, 
	\[ \int_{A_0} \int_{N_{w^{-1}}(\A)\bs N_0(\A)} \int_{\Re \lambda = \lambda_0} \phi(\lambda)(n_w u) \left( a^{-1} n_w \theta_0(a) n_w^{-1} \right)^\lambda \chi(a) \delta_{w^{-1}}(a) \delta_0(a^{-1} n_w \theta_0(a) n_w^{-1})^{\frac{1}{2}} \d \lambda \d u \d a, \]
	where we have used the notation $a^\lambda$ to denote $\exp \sprod{\lambda}{H_0(a)}$ as defined in \cref{mellin}. By \cref{modulus} this reduces to 
	\[ \int_{A_0} \int_{N_{w^{-1}}(\A)\bs N_0(\A)} \int_{\Re \lambda = \lambda_0} \phi(\lambda)(n_w u) \left( a^{-1} n_w \theta_0(a) n_w^{-1} \right)^\lambda \chi(a) \delta_{w^{-1}}(a \theta_0(a^{-1})) \delta_{M_0,N_{w^{-1}}}(a^{-1} \theta_0(a)) \d \lambda \d u \d a. \]
	Note that
	\begin{align*}
		\left( a^{-1} n_w \theta_0(a) n_w^{-1} \right)^\lambda & = \exp\sprod{\lambda}{H_0(a^{-1} n_w \theta_0(a) n_w^{-1})} \\ 
			&= \exp \sprod{\lambda}{-H_0(a) + w.\theta_0(a)} \\
			&= \exp -\sprod{(1-\theta_0^{-1} w^{-1}) \lambda}{H_0(a)}.
	\end{align*}
	Replacing the integral over $A_0$ above with $\aaa_0$ and using \cref{modulus0} reduces the expression to 
	\begin{flalign*}
		\int_{\aaa_0} \int_{N_{w^{-1}}(\A)\bs N_0(\A)} \int_{\Re \lambda = \lambda_0} \phi(\lambda)(n_w u) \exp \sprod{X - w \theta_0(X)}{\lambda} \tau_Q(X - T)\ \hat \tau_{P_0}^Q(T - X) \tau_{P_0}^Q(X-T_1) \\ 
	 \times \sum_{\alpha\in \Delta_Q}  \sprod{X_Q-T_Q}{\alpha}^l \exp \sprod{X - \theta_0(X)}{\frac{1}{2} \sum_{\alpha , w \alpha > 0 } \alpha - \frac{1}{2} \sum_{\substack{\beta > 0 \\ w \beta < 0} } \beta} \d \lambda \d u \d X
	\end{flalign*}

	\begin{flalign*}
		= \int_{\aaa_0} \int_{N_{w^{-1}}(\A)\bs N_0(\A)} \int_{\Re \lambda = \lambda_0} \phi(\lambda)(n_w u) \exp \sprod{X}{-(\lambda - \theta_0^{-1} w^{-1} \lambda) + \frac{1}{2} (1 - \theta_0^{-1})(\sum_{\alpha , w \alpha > 0 } \alpha - \sum_{\substack{\beta > 0 \\ w \beta < 0} } \beta) } \\
		\times \ \tau_Q(X - T)\ \hat \tau_{P_0}^Q(T - X) \tau_{P_0}^Q(X-T_1) \sum_{\alpha\in \Delta_Q}  \sprod{X_Q-T_Q}{\alpha}^l  \d \lambda \d u \d X. 
	\end{flalign*}
			
	Recall from \cref{eta} that the second term in the inner product is $-\eta(\lambda, w, \theta_0)$. We will eventually prove the absolute convergence of this triple integral which will justify the changing the order of integration. Using \cref{moderategrowth}, the function
	\[ h(\lambda) = \displaystyle m(w^{-1}, \lambda) \prod_{\substack{ \alpha \in \Delta_0 :\\ w^{-1} \alpha < 0}}\sprod{\lambda}{\alpha^\vee} \]
	is of moderate growth so we are reduced to proving the absolute convergence of 
	\[ \int_{\Re \lambda = \lambda_0} \displaystyle \frac{h(\lambda)\varphi_{T, Q, l}(\lambda)}{\prod_{\substack{\alpha \in \Delta_0 :\\ w^{-1} \alpha < 0}} \sprod{\lambda}{\alpha^\vee}} \d\lambda, \]
	where
\[ \varphi_{T, Q, l}(\lambda) := \int_{X \in \aaa_0} \exp \sprod{X}{-\eta(\lambda, w, \theta_0) } \tau_Q(X - T) \hat \tau_{P_0}^Q(T - X) \tau_{P_0}^Q(X-T_1) \sum_{\alpha\in \Delta_Q}  \sprod{X_Q-T_Q}{\alpha}^l \d X, \]
	is absolutely convergent by \cref{lemma46}. Having proven \cref{lemma46} which is the twisted equivalent of \cite{FL16}*{Lemma 3.5}, we are in a position to apply \cite{FL16}*{Proposition 3.4} from which the required estimate follows. 
	
\end{proof}


\section{Root Cone Lemma} \label{section:RCL}
	This section is devoted to proving \cref{RCL}, the Root Cone lemma in various cases. It is clear that the radical of $G$ plays no role in the statement of the lemma so we may as well assume that $G$ is semisimple. After reducing to the case when $G$ is simple, we do a case-by-case exhaustion in the case when $G$ is connected split simple wherein, automorphisms of $G$ correspond to those of the Dynkin diagram of $G$. We use the Cartan-Killing classification to enumerate all automorphisms and prove the lemma in each case. The proof for the automorphism of $E_6$ was done using the software {\tt SageMath} \cite{sagemath}. The general case when $G$ is connected simple but possibly not split eludes a proof. 
	
\subsection{Reduction to the simple case}
	Recall that $G$ is called simple (resp. almost-simple) if it is semisimple, noncommutative and every proper normal subgroup is trivial (resp. finite). By the structure theory of connected semisimple groups (note that $G$ is connected), we have an isogeny of $G$ with the product of its minimal almost-simple connected normal subgroups. The automorphism $\theta$ would then be a a composition of two automorphisms, namely a permutation of these factors, and an automorphism of each of these factors. We first show the RCL holds when $\theta$ is a permutation of these factors. This will reduce the problem of proving RCL for split semisimple groups to split almost-simple groups. There is no loss in generality in assuming that the subgroups are all isomorphic (denoted as $H$) and the permutation $\theta$ is a cycle of length $d$. 

\begin{lemma} \label{rcl_simple}
	Let $H$ be a connected reductive group. The Root Cone \cref{RCL} holds when $G \cong H \times \cdots \times H$ ($d$ copies) and $\theta$ is a $d$-cycle that permutes the $d$-copies of $H$. 
\end{lemma}

\begin{proof}
	We will identify $d$ copies of `objects' of $H$ with the corresponding copies in $G$. For instance, suppose $w = (w_1, \cdots, w_d) \in W^G$ is given so that $w_i \in W^H$ for $i \in [d]$. We need to show the existence of $\lambda = (\lambda_1, \cdots, \lambda_d) \in (\aaa_0^G)^*$ so that $\lambda - \theta^{-1}w^{-1} \lambda$ is a positive linear combination of co-roots $\beta^\vee$ whenever $\beta^\vee \in \Delta_0^{Q^G(w)}$. There is no loss in generality to use $\theta, w$ instead of the notationally cumbersome $\theta^{-1}, w^{-1}$. Then, 
	\[ \lambda - \theta w\lambda = (\lambda_1 - w_2 \lambda_2, \lambda_2 - w_3 \lambda_3, \cdots, \lambda_{d-1} - w_d \lambda_d, \lambda_d - w_1 \lambda_1). \]
	Choosing coordinates $\sprod{\lambda_i}{\varpi_\beta^\vee}$ for each $i \in [d]$ and $\beta \in \Delta_0^H$ will define the vectors $\lambda_1, \cdots, \lambda_d \in ((\aaa_0^H)^*)^+$. Fix $\beta \in \Delta_0^H$; we have three cases. 
	\begin{itemize}
	
		\item Suppose there are $i, j \in [d]$ such that $\beta \in \Delta_0^{Q^H(w_j)}\setminus \Delta_0^{Q^H(w_i)}$. Choose 
	\[ \sprod{\lambda_{i+1}}{\varpi_\beta^\vee} > \sprod{\lambda_{i+2}}{\varpi_\beta^\vee} > \cdots > \sprod{\lambda_i}{\varpi_\beta^\vee}. \]
	For any $i' \in [d]\setminus \{i\}$, reading $i'+1$ modulo $d$, we have
	\[ \sprod{\lambda_{i'} - w_{i'+1} \lambda_{i'+1}}{\varpi_\beta^\vee} = \sprod{\lambda_{i'} - \lambda_{i'+1}}{\varpi_\beta^\vee} + \sprod{\lambda_{i'+1} - w_{i'+1}\lambda_{i'+1}}{\varpi_\beta^\vee} > 0.\]
	The former term is positive by the choice above and the latter is non-negative by \cite{MR1890629}*{Ch. VI \S 1.6 Proposition 18}. 
	
	\item Suppose that $\beta \in \cap_{i=1}^d \Delta_0^{Q^H(w_i)}$ then choose 
	\[ \sprod{\lambda_1}{\varpi_\beta^\vee} = \sprod{\lambda_2}{\varpi_\beta^\vee} = \cdots = \sprod{\lambda_d}{\varpi_\beta^\vee} >0. \]
	Since $\sprod{\lambda_{i-1} - w_i \lambda_i}{\varpi_\beta^\vee} = \sprod{\lambda_{i} - w_i \lambda_i}{\varpi_\beta^\vee},$ positivity follows from lemma 2.2 of \cite{FL11a} (whose proof is an application of loc.~cit). 
	
	\item Finally, if $\beta \not \in \cap_{i=1}^d \Delta_0^{Q^H(w_i)}$ then $\lambda_i$ can be choses such that $\sprod{\lambda_i}{\varpi_\beta^\vee}$ is positive. 
	
	\end{itemize}
\end{proof}

We could now assume that $G$ is almost-simple but since the RCL is a statement about the root system of $G$, we may assume that $G$ is simple. Additionally when $G$ is split, the statement reduces to the automorphisms of Dynkin diagrams of simple Lie algebras. The Dynkin diagrams for the families $B_n$ and $C_n$ as well as the exceptional ones $E_7, E_8, F_4, G_2$ have no nontrivial automorphisms. The $A$-type has a unique automorphism and so does the $D$-type when $n \neq 4$. $D_4$ has a nontrivial automorphism of order 3 whereas $E_6$ has an involution. We handle each of these cases below by explicitly constructing root cones for different elements $w \in W$. 

\subsection{Root Cone lemma for type $A_{n-1}$}

Following \cite{MR1153249}*{\S 15.1}, we can explicitly write a basis for the roots and weights as follows. 
The space $\aaa_0$ of roots is spanned by $L_1, \cdots, L_n$ such that $\sum_{i=1}^n L_i = 0$. Simple roots and weights can be respectively taken as
\[ \Delta_0 = \{ \alpha_1, \cdots, \alpha_{n-1} : \alpha_i = L_i - L_{i+1} \} \]
and 
\[ \hat \Delta_0 = \{ \varpi_1, \cdots, \varpi_{n-1} : \varpi_i = L_1 + \cdots + L_i \}. \]
A vector 
\[ \lambda = \sum_{i=1}^{n-1} a_i \varpi_i = \sum_{i=1}^{n-1} \left(\sum_{j=i}^{n-1} a_j\right) L_i =: \sum_{b=1}^{n-1} b_i L_i\]
is in the positive Weyl chamber precisely when $a_1, a_2, \cdots, a_{n-1} > 0$ or equivalently, if 
\[ b_1 > b_2 > \cdots > b_{n-1} >0. \]
We can and will write $\lambda$ as a sum $b_1 L_1 + \cdots + b_{n-1} L_{n-1} + b_n L_n$ with $b_n = 0$. 

The action of any $w \in W \cong S_n$ is by permuting the indices and $\theta_0^{-1} = \theta_0$ acts by $\theta_0(L_i) = - L_{n-i}$ which up to a sign, is the action of the long element $w_0 \in W$. Denoting by $\tau \in S_n$ the action of $w_0 w^{-1}$, we see that 
\[ \theta_0^{-1} w^{-1} \lambda = - w_0 w^{-1} \lambda = - \sum_{i=1}^n b_i L_{\tau(i)} = \sum_{i=1}^n -b_{\tau^{-1}(i)} L_i. \]
Thus, 
\[ \lambda - \theta_0^{-1} w^{-1} \lambda = \sum_{i=1}^n \left( b_i + b_{\tau^{-1}(i)} \right) L_i 
	= \sum_{i=1}^{n-1} (b_i + b_{\tau^{-1}(i)} - b_{\tau^{-1}(n)}) L_i \]
We now use the Cartan matrix to write this vector in terms of the roots. 
\[ 	Z = C U = \begin{pmatrix}
  		2 & -1 & 0 & \cdots & 0\\
		-1 & 2 & -1 & \cdots & 0 \\
		0 & \ddots & \ddots & \ddots & \vdots \\
		\vdots & & -1 & 2 & -1 \\
		0 & \cdots & 0 & -1 & 2
	\end{pmatrix} 
	\begin{pmatrix}
		 1  & 0  & 0  & \cdots  & 0  \\
		 1  & 1  & 0  & \cdots  & 0  \\
		 1  & 1 & 1  &   &  \vdots \\
		 \vdots   & \vdots  & \vdots  & \ddots  &  0 \\
		 1  & 1  &  1 & 1  & 1 
	\end{pmatrix} =
	\begin{pmatrix}
		1 & -1 & 0 & \cdots & 0 \\
		0 & 1 & -1 & \cdots & 0 \\
		\vdots & \vdots & \ddots & \ddots & \vdots \\
		&  &  &  1 & -1 \\
		1 & 1 & \cdots & 1 & 2 						
	\end{pmatrix}.
\]
Since $Z^{-1} L_i = \alpha_i$ for $i = 1, 2, \cdots, n-1$ we have

\[ \lambda - \theta_0^{-1} w^{-1} \lambda = \sum_{i=1}^{n-1} (b_i + b_{\tau^{-1}(i)} - b_{\tau^{-1}(n)}) L_i  = \sum_{i=1}^{n-1} c_i \alpha_i \]
where 
\begin{flalign*}
	c_i & = \displaystyle \frac{1}{n} \left[ (n-i) \sum_{j=1}^i (b_j + b_{\tau^{-1}(j)} - b_{\tau^{-1}(n)}) - i \sum_{j=i+1}^{n-1} (b_j +b_{\tau^{-1}(j)} -b_{\tau^{-1}(n)} ) \right] \\
		& = \displaystyle \frac{1}{n} \left[ n \sum_{j=1}^i (b_j + b_{\tau^{-1}(j)} - b_{\tau^{-1}(n)}) - i \sum_{j=1}^{n-1} (b_j +b_{\tau^{-1}(j)} -b_{\tau^{-1}(n)} ) \right] \\
		\therefore nc_i & = n(b_1 + \cdots + b_i) -i (b_1 \cdots + b_n) + n(b_{\tau^{-1}(1)} + \cdots b_{\tau^{-1}(i)}) - i (b_{\tau^{-1}(1)} + \cdots + b_{\tau^{-1}(n)}) \\
		& = n(b_1 + \cdots + b_i + b_{\tau^{-1}(1)} + \cdots + b_{\tau^{-1}(i)} ) -2i (b_1 + \cdots + b_n).
\end{flalign*}

Thus we need to be able to choose coefficients $b_i$ such that $c_i$ above is positive whenever $\alpha_i \in \Delta_0^{Q(w)}$. By \cite{FL11a}*{p. 787}, 
\[ \Delta_0^{Q(w)} = \{ \alpha \in \Delta_0 : w \varpi_\alpha^\vee \neq \varpi_\alpha^\vee \}, \]
which for $\GL(n)$ implies that $\alpha_i \in \Delta_0^{Q(w)}$ precisely if the permutation $\tau$ above satisfies $\tau([i]) \neq [i] + n-i, \text{ where } [n] := \{ 1, 2, \cdots, n\}$, cf. \cite{mathoverflow}. This is proven in the following 

\begin{lemma} \label{majer}
	For $n \geq 2$, fix a permutation $\tau \in S_n, \tau \neq (1,n)(2,n-1)\cdots$, i.e., $\tau$ isn't the long element. Let 
$$ \Delta(\tau) = \{ i \in [n-1] : \tau([i]) \neq \{ [i]+n-i \} \}.$$
Then there exist real numbers $b_1 > b_2 > \cdots > b_{n-1} > b_n = 0$ such that the inequalities 

\begin{equation}
	\frac{b_1 + b_2 + \cdots + b_i + b_{\tau^{-1}(1)} + \cdots + b_{\tau^{-1}(i)}}{2i} > \frac{b_1 + b_2 + \cdots + b_n}{n} 
\end{equation}
hold simultaneously for every $i \in \Delta(\tau)$.
\end{lemma}

\begin{proof}

For a given $\tau\in S_n$ and for any $j\in[n]$ define the numbers
$$a_j:=\chi_{\Delta}(j)-\chi_{\Delta}(j-1)-\chi_{\Delta }(n-j)+\chi_{\Delta}(n-j+1),$$
where $\chi_\Delta:\mathbb{Z}\to\{0,1\}$  denotes the characteristic function of  the set $\Delta:=\Delta(\tau)\subset\mathbb{Z}$. Also, with $c:=5n-a_n$, define
$$b_j:=a_j-5j+c.$$
We may note right away that since $|a_j|\le2$, the $b_j$ are strictly decreasing, and that $b_n=0$, by the choice of the constant $c$.

For any $E\subset[n]$, for simplicity of notation we put
 $$ \alpha(E):=\sum_{j\in E} a_j,\qquad \beta(E):=\sum_{j\in E} b_j$$
(so we may think $ \alpha$ and $\beta$ as discrete signed measures supported in $[n]$).

For  $i\in[n]$, summing over $j=1,\dots i$ we have
$$ \alpha([i])=\chi_\Delta(i)-\chi_\Delta(n-i).$$

Incidentally, for any $i\in[n]$ we have $i\in\Delta$ if and only if, by definition, $\tau([i])\neq[i]+n-i$ thus also, since $\tau$ is bijective,  if and only if $\tau([i]^c)\neq([i]+n-i)^c$, that is  $\tau([n-i]+i)\neq [n-i]$ or $\tau^{-1}([n-i])\neq [n-i]+i$, which means $n-i \in \Delta^{-1}:=\Delta(\tau^{-1})$. Hence the last formula also writes
$$ \alpha([i])=\chi_{\Delta}(i)-\chi_{\Delta^{-1}}(i).$$
Also note that, since $n\not\in\Delta$
$$ \alpha([n])=0,$$
and $$ \alpha([i]+n-i)=- \alpha([n-i]) =-\chi_\Delta(n-i)+\chi_\Delta(i)= \alpha([i]).$$

We proceed showing the inequalities on the arithmetic means. 

Case I. Assume $i\in\Delta\setminus\Delta^{-1}$. Then by definition of $\Delta^{-1}$,  $\tau^{-1}([i])=[i]+n-i$, so that

\[ \frac{\alpha([i])+ \alpha(\tau^{-1}[i])}{2i} = {\alpha([i])+ \frac{\alpha([i]+n-i)}{2i}} = \frac{\chi_{\Delta}(i)-\chi_{\Delta^{-1}}(i)}{i} = \frac{1}{i} > 0, \]
and summing the arithmetic means of $-5j+c$ on the same sets we have plainly
$$ \frac{\beta([i])+ \beta(\tau^{-1}[i])}{2i} > \frac{\beta([n])}{n}.$$

Case II. Assume $i\in\Delta\cap\Delta^{-1}$. Thus $\tau^{-1}([i])\neq[i]+n-i$ and, just because $b_j$ are strictly decreasing

$$ \frac{\beta([i])+ \beta(\tau^{-1}[i])}{2i} > \frac{\beta([i])+ \beta([i]+n-i)}{2i} $$
and since we have $\alpha([i])=\alpha([i]+n-i)=\alpha([n])=0$ because $\chi_{\Delta}(i)=\chi_{\Delta^{-1}}(i)=1$,  summing as before the arithmetic means of the affine part of $b_j$, $$ \frac{\beta([i])+ \beta([i]+n-i)}{2i} = \frac{\beta([n])}{n},$$
concluding the proof.

\end{proof}

\subsection{Root Cone lemma for type $D_\l$}

The root system $D_\l$ has a unique automorphism when $\l \neq 4$ which for $\l = 3$ coincides with the automorphism $\theta_0(x) = {}^t\!x^{-1}$ of $GL(4)$ (and this case has been proven in the previous section). The root system $D_4$ has an automorphism of order 3 which will be considered later. 

Following \cite{MR1890629} we assume the ambient space is spanned by the vectors $\{e_1, \cdots, e_\l\}$ and the roots are given by $\RRR_0 = \{ \pm e_i \pm e_j : 1 \leq i < j \leq \l \}$. For a base $\Delta_0$, we choose
\[ \Delta_0 = \{ \alpha_1 = e_1 - e_2, \cdots, \alpha_{\l - 1} = e_{\l - 1} - e_{\l}, \alpha_{\l} = e_{\l - 1} - e_{\l} \}.  \]
The corresponding dual basis of weights is
\[ \hat{\Delta}_0 = \{ \varpi_i = e_1 + \cdots + e_i : 1 \leq i \leq \l - 2 \} \cup \{ \varpi_{\l - 1} = \frac{1}{2}(e_1 + \cdots + e_{\l - 1} -e_\l), \varpi_\l = \frac{1}{2}(e_1 + \cdots + e_\l) \}. \]
Since the root system $D_\l$ is selfdual, the co-roots and co-weights are defined similarly. The Weyl group $W$ consists of even-signed permutations, i.e., any $w \in W$ is a pair $(\sigma, \eta)$ where $\sigma \in S_\l$ and $\eta = (\eta_1, \cdots, \eta_\l)$ is an ordered $\l$-tuple of $\pm 1$ with even $-1$'s. The action of $w$ on the basis is given by 
\[ w.e_i = \eta_i e_{\sigma(i)}. \]
The automorphism $\theta_0$ acts on $\Delta_0$ by permuting the set $\{ \alpha_{\l - 1}, \alpha_\l \}$ and fixing other roots and similarly on the weights in $\hat\Delta_0$. 

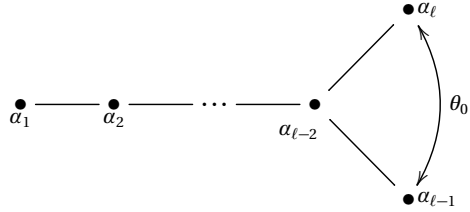
\[ \xymatrix{
	& & & & \bullet \ar @{}_{\alpha_\l} \\ 
	\bullet \ar @{-} [r] \ar@{}_{\alpha_1} & \bullet \ar @{-} [r] \ar@{}_{\alpha_2} & \cdots \ar @{-} [r]  & \bullet \ar @{-} [dr] _(0.13){\alpha_{\l-2}} \ar @{-} [ur] \\
	& & & & \bullet \ar @/_1pc/ @{<->} [uu] _{\theta_0} \ar @{}_{\alpha_{\l-1}}
} \]

To prove the Root Cone lemma for $D_\l$, we need to show that for fixed $w = (\sigma, \eta) \in W$ there exists an open cone $\Omega \subseteq (\aaa_0^*)^+$ such that if $\lambda \in \Omega$, the inequality 

\refstepcounter{equation}\label{eq:so2n}	

\begin{equation}
	\sprod{\lambda - \theta_0^{-1} w^{-1} \lambda}{\varpi_i^\vee} = \sprod{\lambda}{\varpi_i^\vee - w \theta_0 \varpi_i^\vee} > 0 
    \tag*{\myTagFormat{eq:so2n}{i}}\label{eq:so2n-i} 
\end{equation}

holds whenever $w \varpi_i^\vee \neq \varpi_i^\vee$. Assume that $\lambda = c_1 \varpi_1 + \cdots + c_\l \varpi_\l \in (\aaa_0^*)^+$ that is to say, $c_i > 0$ for all $i \leq \l$. Observe that 
\begin{multline} \label{lambdacoeff}
	\lambda = c_1 e_1 + c_2 (e_1 + e_2) + \cdots + c_{\l - 1} \left(e_1 + \cdots + e_{\l - 1} + 
			\frac{e_{\l - 1} + e_{\l}}{2}\right) \\ + c_\l \left(e_1 + \cdots + e_{\l - 2} + \frac{-e_{\l - 1} + e_\l}{2}\right) \\ 
			= (c_1 + \cdots + c_{\l - 2}) e_1 + \cdots + \left(c_{\l - 2} + \frac{c_{\l-1} + c_\l}{2}\right) e_{\l - 2} \\ + \left(\frac{c_{\l-1} + c_\l}{2}\right) e_{\l - 1} + \left(\frac{-c_{\l-1} + c_\l}{2}\right) e_\l, 
\end{multline}
so if $c_1, \cdots, c_\l > 0$ then the only possible negative coefficient above is that of $e_\l$. If $\theta_0$ fixes $\varpi_i^\vee$ then every $\lambda \in (\aaa_0^*)^+$ satisfies the Inequality~\myTagFormat{eq:so2n}{i}, as can be seen from \cite{FL11a}*{Lemma 2.2}. It suffices to prove these inequalities for $i = \l - 1, \l$. We first analyze the case $i = \l - 1$ and set $I(w) = \{ \sigma(i) : 1 \leq i \leq \l , \ \eta_i = -1 \}$. 

\begin{flalign*}
	\varpi_{\l - 1}^\vee - w \theta_0 \varpi_{\l - 1}^\vee & = \varpi_{\l - 1}^\vee - w \varpi_\l^\vee \\
		& = \frac{1}{2} (e_1^\vee + \cdots + e_{\l - 1}^\vee - e_\l^\vee) - \frac{1}{2}(\eta_1 e_{\sigma(1)}^\vee + \cdots + \eta_\l e_{\sigma(\l)}^\vee) \\
		& = \left( \sum_{ i \in I(w) } e_{\sigma(i)}^\vee \right) - e_\l^\vee .
\end{flalign*}

\begin{itemize}
	\item If $I(w) = \emptyset$ then $\eta = (1, \cdots, 1)$ and $\varpi_{\l - 1}^\vee - w\varpi_\l^\vee = -e_\l^\vee$ so to ensure the quantity $\sprod{\lambda}{\varpi_{\l - 1}^\vee - w\varpi_\l^\vee}$ is positive, by \cref{lambdacoeff} we can take $c_{\l - 1} > c_\l$. 
	
	\item If $I(w) \neq \emptyset$ and $\inf I(w) < \l - 1$, we can choose $c_i\gg c_{\l - 1}, \ c_i \gg c_\l$ for some $i \in I(w) \setminus \{ \l - 1, \l \}$ to get the positivity condition. 
	
	\item Finally if $\inf I(w) \geq \l - 1$ then $\eta$ must have exactly two sign changes and $I(w) = \{ \l - 1, \l \}$ in which case $\varpi_{\l - 1}^\vee - w\varpi_\l^\vee = e_{\l - 1}^\vee$ so every $\lambda \in (\aaa_0^*)^+$ satisfies Inequality~\myTagFormat{eq:so2n}{l-1}. 
\end{itemize}

The case $i = \l$ is slightly more involved. Observe that if $w \varpi_\l^\vee \neq \varpi_\l^\vee$ then $\eta \neq (1, \cdots, 1)$ where $w = (\sigma, \eta)$. Thus we can assume $I(w) \neq \emptyset$. 

\begin{flalign*}
	\varpi_\l^\vee - w \theta_0 \varpi_\l^\vee & = \varpi_\l^\vee - w\varpi_{\l - 1}^\vee \\
		& = \frac{1}{2}(e_1^\vee + \cdots + e_\l^\vee) - \frac{1}{2} (\eta_1 e_{\sigma(1)}^\vee + \cdots + \eta_{\l - 1} e_{\sigma(\l - 1)}^\vee - \eta_\l e_{\sigma(\l)}^\vee) \\
		& = \frac{1}{2} \left( e_1^\vee + \cdots + e_\l^\vee - \eta_1 e_{\sigma(1)}^\vee - \cdots - \eta_\l e_{\sigma(\l)}^\vee \right) + \eta_\l e_{\sigma(\l)}^\vee \\
		& = \eta_\l e_{\sigma(\l)}^\vee + \sum_{i \in I(w)} e_i^\vee. 
\end{flalign*}
Thus $\varpi_\l^\vee - w \theta_0 \varpi_\l^\vee$ is a positive linear combination of $\{e_i\}$ whenever $\eta_\l \neq -1$. If $\eta_\l = -1$ then $\eta_{i_0} = -1$ for some $i_0 \neq \l$;
\begin{flalign*}
	\varpi_\l^\vee - w\varpi_{\l - 1}^\vee & = - e_{\sigma(\l)}^\vee + e_{\sigma(\l)}^\vee + e_{\sigma(i_0)}^\vee + \sum_{i \in I(w) \setminus \{ i_0, \l \}} e_i^\vee \\
		& = e_{i_0}^\vee + \sum_{i \in I(w) \setminus \{ i_0, \l \}} e_i^\vee,
\end{flalign*}
which again is a positive linear combination of $\{e_i\}$'s. In either case, choosing $c_{\l - 1} < c_\l$ ensures that Inequality~\myTagFormat{eq:so2n}{l} holds. Observe that this choice is consistent with the choices in Inequality~\myTagFormat{eq:so2n}{l-1}. \qed

\subsection{The triality automorphism of $D_4$}

Assume that $\theta_0$ is the automorphism of $D_4$ permuting the set $\{ \alpha_1, \alpha_3, \alpha_4 \}$ cyclically as shown below. 

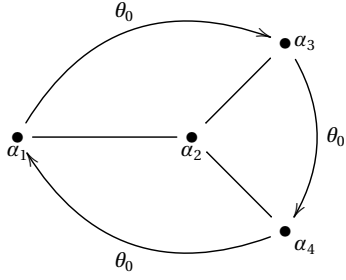
\[ \xymatrix{ 
	& & & \bullet \ar@{}_{\alpha_3} \ar @/^1pc/ [dd] ^{\theta_0} \\
	\bullet \ar @/^2pc/ [urrr] ^{\theta_0} \ar@{-} [rr] \ar@{}_(-0.5){\alpha_1} & & \bullet \ar@{}_(0){\alpha_2} \ar @{-} [dr] \ar @{-} [ur] & \\
	& & & \bullet \ar@{}_{\alpha_4} \ar @/^2pc/ [ulll] ^{\theta_0}
} \]

We need to find conditions on $c_1, c_2, c_3, c_4 > 0$ where 
\[ \lambda = \left( c_1 + c_2 + \frac{c_3 + c_4}{2} \right) e_1 + \left( c_2 + \frac{c_3 + c_4}{2} \right) e_2 + \left( \frac{c_3 + c_4}{2} \right) e_3 + \left( \frac{-c_3 + c_4}{2} \right) e_4 \]
so that 
\begin{enumerate}[(a)]
	\item \label{lambdaa} $\sprod{\lambda}{\varpi_4^\vee - w \varpi_1^\vee} > 0 $ whenever $w \varpi_4^\vee \neq \varpi_4^\vee$;
	\item \label{lambdab} $\sprod{\lambda}{\varpi_3^\vee - w \varpi_4^\vee} > 0 $ whenever $w \varpi_3^\vee \neq \varpi_3^\vee$; and 
	\item \label{lambdac} $\sprod{\lambda}{\varpi_1^\vee - w \varpi_3^\vee} > 0 $ whenever $w \varpi_1^\vee \neq \varpi_1^\vee$. 
\end{enumerate}
Let us analyze the inequality in each case. 
\begin{enumerate}[(a)]
	\item Observe that 
		\[ \varpi_4^\vee - w \varpi_1^\vee = \frac{1}{2}(e_1^\vee + e_2^\vee + e_3^\vee + e_4^\vee) - \eta_1 e_{\sigma(1)}^\vee. \]
    	\begin{itemize}
    		\item If $\eta_1 = -1$ then this vector is a positive combination of the $e_i^\vee$'s so $\sprod{\lambda}{\varpi_4^\vee - w \varpi_1^\vee} > 0$ can be ensured by choosing $c_3 < c_4$. 
    		\item If $\eta_1 = 1$ and $\sigma(1) \in \{ 2, 3, 4 \}$ then choosing $c_1 \gg c_3, c_4$ guarantees Inequality~(\ref{lambdaa}) holds. 
    		\item Finally if $\eta_1 = 1$ and $\sigma(1) = 1$ then
    			\[ \sprod{\lambda}{\varpi_4^\vee - w \varpi_1^\vee} = \sprod{\lambda}{\frac{1}{2}(-e_1^\vee + e_2^\vee + e_3^\vee + e_4^\vee)} = -c_1 + c_4. \]
    			Pick $c_1 < c_4$. Observe that if $\eta_1 = 1$ and $\sigma(1) = 1$ then $w \varpi_1^\vee = \varpi_1^\vee$ so Inequality~(\ref{lambdac}) need not be verified. 
    	\end{itemize}		
	\item Writing $\sprod{\lambda}{\varpi_3^\vee - w \varpi_4^\vee}$ as a sum of $\sprod{\lambda}{\varpi_4^\vee - w \varpi_4^\vee}$ (which is positive by \cite{FL11a}*{Lemma 2.2}) and $\sprod{\lambda}{\varpi_3^\vee - \varpi_4^\vee} = \sprod{\lambda}{-e_4^\vee} = \frac{-c_3 + c_4}{2}$, we can ensure Inequality~(\ref{lambdab}) holds if $c_3 < c_4$. 
	\item The Inequality~(\ref{lambdac}) is proven similarly;
		\[ \sprod{\lambda}{\varpi_1^\vee - w \varpi_3^\vee} = \sprod{\lambda}{\frac{1}{2}(e_1^\vee - e_2^\vee - e_3^\vee + e_4^\vee)} = \frac{c_1 - c_3}{2}, \]
		so choose $c_1 > c_3$. 
\end{enumerate}
It is easy to verify that choices made above are consistent with each other. \qed

\subsection{Root Cone lemma for type $E_6$}
	Following notations of \cite{MR1890629}, the automorphism of $E_6$ is shown below. 
	
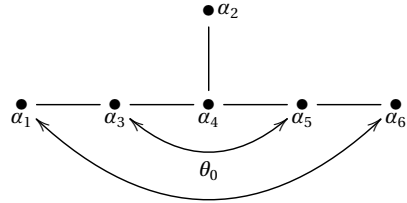
\[ \xymatrix{ 
	& & \bullet \ar@{}_{\alpha_2} & & \\
	\bullet \ar@{-}[r] \ar@{}_{\alpha_1} \ar@{<->}@/_3pc/[rrrr] & \bullet \ar@{-}[r] \ar@{}_{\alpha_3} \ar@{<->}@/_1.5pc/[rr]_{\theta_0} & \bullet \ar@{-}[r] \ar@{}_{\alpha_4} \ar@{-}[u] & \bullet \ar@{-}[r] \ar@{}_{\alpha_5} & \bullet \ar@{}_{\alpha_6} \ar@{}[d]\\
	& & & & 
	} \]

The {\tt SageMath} code below finds a point in the positive Weyl chamber satisfying \cref{RCLinequality} in \cref{RCL} thereby proving the set of solutions is nonempty and open. \\

\begin{Verbatim}
# Proof of the root cone conjecture for the (unique) automorphism of the Dynkin Diagram 
of E_6 using "SageMath".

R = RootSystem([`E', 6]);
X = R.root_space()
W = X.weyl_group()
alpha = X.basis()
alphacheck = X.coroot_space().basis()
varpi = X.fundamental_weights_from_simple_roots()
varpicheck = X.coroot_space().fundamental_weights_from_simple_roots()

def theta(vector):
    sigma = PermutationGroupElement(`(1,6)(3,5)(2)(4)')
    for i in range(1, len(varpi)+1):
        if vector == varpi[i]:
            break
    return varpi[sigma(i)]

def delta(w):
    list = []
    for item in varpi:
        if (w.action(item)).to_ambient() != item.to_ambient():
            list.append(item)
    return list

def rhs(w, vector):
    return (vector - w.action(theta(vector))).to_ambient()

def isPositive(Lambda, w):
    for vector in delta(w):
        if (Lambda.to_ambient()).dot_product(rhs(w, vector)) <= 0:
            return False
    return True

def isSuccess(w):
    for x_1 in range(1,4):
        for x_2 in range(1,4):
            for x_3 in range(1,4):
                for x_4 in range(1,4):
                    for x_5 in range(1,4):
                        for x_6 in range(1,4):
                            Lambda = x_1 * varpicheck[1] + x_2 * varpicheck[2] 
                            + x_3 * varpicheck[3] + x_4 * varpicheck[4] + 
                            x_5 * varpicheck[5] + x_6 * varpicheck[6]
                            if isPositive(Lambda, w):
                                return (Lambda, True)
    Lambda = 0 * varpicheck[1]
    return (Lambda, False)
    
count = 0
for w in W:
    if isSuccess(w)[1] == True:
        count = count + 1
    else:
        print "Root Cone Conjecture fails for w = ", w
print "Number of successful elements = ", count, "(Order of W =", W.cardinality(), ")"

\end{Verbatim}

	
\section{The spectral side} \label{section:spec}
	
	We begin by reviewing the spectral side of the twisted trace formula. 	
	
	\subsection{Twisted $(G,M)$-families}
		
		Recall the definition of a $(G,M)$-family \cite{MR625344}*{\S 6}, namely a collection 
		\[ c(\Lambda, P), \ P \in \PPP(M), \Lambda \in i\aaa_M^* \]
		of smooth functions is called a $(G, M)$-family if 
		\[ c(\Lambda, P) = c(\Lambda, P') \]
		for any pair $P, P' \in \PPP(M)$ of adjacent groups and any point $\Lambda$ in the hyperplane spanned by the common wall of the chambers corresponding to $P$ and $P'$. It is well-known that a $(G,M)$-family $c(\Lambda, P)$ gives naturally a smooth function of $\Lambda$ as 
		\[ c_M(\Lambda) = c_M^G(\Lambda) := \sum_{P \in \PPP(M)} \epsilon_P(\Lambda) c(\Lambda, P). \]
		For a definition of $\epsilon_P(\Lambda)$, see \cref{epsilon} in \cref{section:notation}. A $(G,M)$-family $c(\Lambda, P)$ also gives a $(\tild G, \tild M)$-family 
		\[ \{ c(\Lambda, \tild P) : \; \Lambda \in i \aaa_{\tild P}^*, \ \tild P \in \PPP(\tild M)\} \]
		by restricting $\Lambda \in i \aaa_P^*$ to the subspace $i \aaa_{\tild P}^*$. The corresponding smooth function is given by 
		\[ c_{\tild M}(\Lambda) = \sum_{\tild P \in \PPP(\tild M)} \epsilon_{\tild P}(\Lambda) c(\Lambda, \tild P). \]

	\subsection{Intertwining operators} 
	
	Let $P$ and $Q$ be associated parabolic subgroups. This means that the set $W(\aaa_P, \aaa_Q)$ of isomorphisms from $\aaa_P$ to $\aaa_Q$ arising from restrictions of elements of $W$ is non-empty. Fix $w \in W(\aaa_P, \aaa_Q)$ and $\nu \in \aaa_{P,\C}^*$. The intertwining operator $M_{Q|P}(w,\nu)$ is defined by 
	\[ M_{Q|P}(w,\nu)\Phi(x) = \exp -\sprod{w\nu+\rho_Q}{H_Q(x)} \int_{N_{w,P,Q}(\A)} \Phi(n_w^{-1}nx) \exp \sprod{\nu+\rho_P}{H_P(n_w^{-1}nx)} \d n, \]
	where $N_{w,P,Q} = N_Q \cap n_w N_P n_w^{-1} \bs N_Q$. It is an operator from $\AAA(X_P)$ to $\AAA(X_Q)$ which maps the subspace $\AAA(X_P, \sigma)$ to $\AAA(X_Q, \theta \circ \sigma)$. The intertwining operator $M(w, \lambda)$ defined in \cref{easyintertwining} corresponds to $P = Q = P_0$ is just a special case of this one. As remarked before, the integral converges only for $\Re(\nu)$ in a certain chamber but $M_{P|Q}(w, \nu)$ can be analytically continued to a meromorphic function of $\nu \in \aaa_{P,\C}^*$. Set $M_{Q|P}(\nu) = M_{Q|P}(1,\nu)$. 
	
	For $P, P_1 \in \PPP(M)$ and $\nu, \Lambda \in i\aaa_P^*$, the collection 
	\[ \MMM(P, \nu; \Lambda, P_1) := M_{P_1|P}(\nu)^{-1} M_{P_1|P}(\nu + \Lambda), \quad P_1 \in \PPP(M), \Lambda \in i\aaa_M^* \]
	is a $(G,M)$-family with corresponding smooth function of $\Lambda$ given by 
	\[ \MMM_M(P, \nu; \Lambda) = \MMM_M^G(P, \nu; \Lambda) := \sum_{P_1 \in \PPP(M)} \epsilon_{P_1}(\Lambda) \MMM(P, \nu; \Lambda, P_1). \]
	As discussed above, we have the associated $(\tild G, \tild M)$-family $\MMM(P, \nu; \Lambda, \tild{P_1}), \ \tild{P_1} \in \PPP(\tild M)$ and the smooth function $\MMM_M(P, \nu; \Lambda)$ where $\Lambda$ belongs to the subspace $i\aaa_{\tild M}^*$ of $i \aaa_M^*$. Set 
	\[ \MMM_{\tild M}(P, \nu) = \MMM_{\tild M}(P, \nu; 0). \]
	
	It is one of the basic properties of $(G,M)$-families that the limit of $\MMM_M(P,\nu; \Lambda)$ as $\Lambda \to 0$ exists. We are now ready to state the twisted trace formula for a test function $h \in \CCC_c^\infty(\tild G(\A))$ as in \cite{LW}*{Theor\`eme 14.3.3}. 
	
	\begin{equation}\label{eq:spectral}
		J^{\tild G}(h, \omega) = \sum_{\tild L \in \LLL^{\tild G}} \frac{\mod{\tild{W}^L}}{\mod{\tild{W}^G}} \frac{1}{\mod{\det{(\theta_L - 1|_{\aaa_M^G/\aaa_{\tild L}^{\tild G}})}}} J_{\tild L}^{\tild G}(h, \omega), 
	\end{equation}
	where 
	\[		J_{\tild L}^{\tild G}(f, \omega) = \sum_{M \in \LLL^L} \frac{\mod{W^M}}{\mod{W^L}} \sum_{\tild w \in W^{\tild L}(M)_{\text{reg}}} \frac{1}{\mod{\det{(\tild w - 1|_{\aaa_M^L})}}} J_{M, \tild w}^{\tild G}(h, \omega),
	\]
	where finally, 
	\[ J_{M, \tild w}^{\tild G}(h, \omega) = \int_{i (\aaa_{\tild L}^{\tild G})^* } \trace\left(\MMM_{\tild L}^{\tild G}(P, \nu) M_{P|\tild w P}(0) \rho_{P, \text{disc}, \nu}(\tild w, h, \omega) \right) \d \nu. \]
	
	The previous expression can also be written as 
	\[ J_{M, \tild w}^{\tild G}(h, \omega) = \sum_{\sigma \in \Pi_\text{disc} (M)} \int_{i (\aaa_{\tild L}^{\tild G})^* } \trace\left(\MMM_{\tild L}^{\tild G}(P, \nu) M_{P|\tild w P}(0) \rho_{P, \sigma, \nu}(\tild w, h, \omega) \right) \d \nu. \]	
	
	The absolute convergence of the spectral side would imply that the distribution $J^{\tild G}(f, \omega)$ extends continuously to $\CCC(\tild G(\A), K)$. We prove it by adopting the method of \cite{FL11b} and \cite{FLM} for the twisted trace formula. In the regular (non-twisted) trace formula, $\MMM_L(P, \nu), M_{P|w.P}(0)$ and $\rho_{P, \nu}(h)$ are operators on the space $\AAA(X_P)$ and the trace of their composition is integrated over the parameter $\nu$. However, the twisted regular representation $\rho_{P, \text{disc}, \nu}(\tild w, f, \omega)$ maps vectors in $\AAA(X_P)$ into those in $\AAA(X_Q)$ where $Q = \theta(P) = \Ad(\tild w)(P)$. The intertwining operator $M_{P|Q}(\nu)$ does the opposite, hence $M_{P|Q}(\nu) \circ \rho_{P, \text{disc}, \nu}(\tild w, f, \omega)$ is an operator on $\AAA(X_P)$. We define a unitary operator $\myshift = \myshift_{P, \nu}(\tild w)$ which satisfies the lemma below. 
	
	\[	\myshift_{P, \nu}(\tild w) : \ \AAA(X_P) \to \AAA(X_Q) \]
	\[ (\myshift_{P, \nu}(\tild w)\Phi)(x) = \exp{-(\sprod{\tild w \nu + \rho_Q}{H_Q(x)})} \Phi(n_{\tild w}^{-1} x n_{\tild w}) \exp{(\sprod{\nu + \rho_P}{H_P(n_{\tild w}^{-1} x n_{\tild w})})}. \]
	Recall above that $n_{\tild w}$ is the representative of $\tild w$ in $\tild G(\Q)$. 
	
	\begin{lemma} \label{lemma:myshift}
		\begin{enumerate}
			\item For any $y \in \tild G(\A), y = n_{\tild w}g $ with $g \in G(\A)$, we have
		\[ \rho_{P, \text{disc}, \nu}(\tild w, y, \omega) = \underline{\omega} \ \myshift_{P, \nu}(\tild w) \ \rho_{P, \nu}(g). \]
				Here, $\underline\omega$ is the operator $(\underline\omega \Phi)(x) = w(x) \Phi(x)$. 
		
			\item \[ \rho_{P, \text{disc}, \nu}(\tild w, f, \omega) = \underline{\omega} \ \myshift_{P, \nu}(\tild w) \ \rho_{P, \nu}(h), \]
			where $h = L_{w^{-1}}f$. 
			
			\item The operator $\myshift_{P, \nu}(\tild w)$ is unitary and invertible. 
		\end{enumerate}
	\end{lemma}
	
	\begin{proof}
		This follows from the definitions and \cref{bijectionlemma}. 
	\end{proof}
	
	\begin{theorem} \label{spectral}
		For any $f \in \CCC^\infty_c(\tild G(\A))$, the spectral side of the twisted trace formula is given by \eqref{eq:spectral}. In this equation the sums are finite (except the one over $\sigma \in \Pi_\text{disc} (M)$) and the integrals are absolutely convergent with respect to the trace norm, hence define distributions on $\CCC(\tild G(\A), K)$. 
	\end{theorem}
	
	\begin{proof}
	
		We fix $\tild L \in \LLL^{\tild G}, M \in \LLL^L, \tild w \in W^{\tild L}(M)_{\text{reg}}$ and show that the integral 
		\[ \int_{i (\aaa_{\tild L}^{\tild G})^*} \norm{\MMM_{\tild L}^{\tild G}(P, \nu) M_{P|\tild w P}(0) \rho_{P, \text{disc}, \nu}(\tild w, h, \omega)} \d \nu \]
		converges, where $\norm{\circ}$ denotes the trace norm of the operator on the space $\overline \AAA(X_P)$. Since the operator $\MMM_{\tild L}^{\tild G}(P, \nu)$ equals $\MMM_L^G(P, \nu)$ on the subspace $i\aaa_{\tild L}$, hence it decomposes into finite sums of the composition of intertwining operators and their first-order derivatives. Referring to notations therein, we recall 
		\begin{theorem}\cite{FLM}*{Theorem 2}
			Let $M \in \LLL, P \in \PPP(M), L \in \LLL(M), m = a_L^G$ and $\underline \mu \in (\aaa_M^*)^m$ be in general position. Then we have 
			\[ \MMM_L(P, \lambda) = \sum_{\underline \beta \in \mathcal B_{P, L}} \Delta_{\mathcal X_{L, \underline{\mu}}(\underline \beta)}(P, \lambda). \] 
		\end{theorem}
		Since the sum is over a finite set, it suffices to prove the convergence, for a fixed $m$-tuple $\mathcal X$ of parabolic subgroups, of 
		\[ \int_{i (\aaa_{\tild L}^{\tild G})^*} \norm{\Delta_\mathcal X(P, \nu) M_{P|\tild w P}(0) \rho_{P, \text{disc}, \nu}(\tild w, h, \omega)} \d \nu. \]
		By \cref{lemma:myshift}, we can write this as 
		\[ \int_{i (\aaa_{\tild L}^{\tild G})^*} \norm{\Delta_\mathcal X(P, \nu) M_{P|\tild w P}(0) \circ \underline \omega\  \myshift_{P, \nu}(\tild w) \rho_{P, \nu}(h)} \d \nu. \]
		If we denote the composite operator $M_{P|\tild w P}(0) \, \underline \omega\  \myshift_{P, \nu}(\tild w)$ by $\UUU$ for convenience, we see that the resulting expression 
		\[ \int_{i (\aaa_{\tild L}^{\tild G})^*} \norm{\Delta_\mathcal X(P, \nu) \ \UUU \ \rho_{P, \nu}(h)} \d \nu \]
		resembles that in \cite{FLM}*{Theorem 3}. The theorem follows by imitating \cite{FLM}*{\S5.1} which we now elaborate. We have the algebraic decomposition 
		\[ \AAA(X_P) = \bigoplus_{\sigma \in \Pi_{\text{disc}}(M(\A))} \AAA(X_P, \sigma) \]
		where $\AAA(X_P, \sigma)$ is the $K$-finite part of $\overline \AAA(X_P, \sigma)$. We further decompose 
		\[ \AAA(X_P, \sigma) = \oplus_{\tau \in \hat \K_\infty} \AAA(X_P, \sigma)^\tau \]
		according to the isotypical subspaces for the action of $\K_\infty$. Let $\AAA(X_P, \sigma)^K$ be the subspace of $K$-invariant functions in $\AAA(X_P)$, and similarly for $\AAA(X_P)^{\tau, K}$ for any $\tau \in \hat \K_\infty$. The above integral reduces to 
		\[ \sum_{\sigma \in \Pi_{\text{disc}}(M(\A))} \sum_{\tau \in \hat \K_\infty} \int_{i (\aaa_{\tild L}^{\tild G})^*} \norm{\Delta_\mathcal X(P, \nu) \ \UUU \ \rho_{P, \nu}(h)}\bigg|_{\AAA(X_P, \sigma)^{\tau, K}} \d \nu. \]

		Observe that $\UUU$ restricts to a unitary operator on the finite dimensional space $\overline \AAA(X_P, \sigma)^{\tau, K}$ of $K$-fixed vectors in the $\sigma$-isotypical subspace of $\overline \AAA(X_P)$ which transform according to $\tau \in \hat \K_\infty$ under $\K_\infty$. The operator norm of the composition of operators is controlled by the norms of the operators. Using this trick in \cite{FLM}*{\S 5.1}, we can replace the test function $h$ by a high enough exponent of the operator $\Delta = \Id - \Omega + 2\Omega_{\K_\infty}$. Replacing $\Delta_{\mathcal X}(P, \nu)$ with its expansion, the integrals equals a constant multiple of 
		\begin{multline*}
			\sum_{\sigma \in \Pi_{\text{disc}}(M(\A))} \sum_{\tau \in \hat \K_\infty} \int_{i (\aaa_{\tild L}^{\tild G})^*} \norm{ M_{P_1|P}(\nu)^{-1} \delta_{P_1| P_1'}(\nu) M_{P_1'|P_2}(\nu) \\ \dots \delta_{P_{m-1}|P_{m-1}'}(\nu) M_{P_{m-1}'|P_m}(\nu) \delta_{P_m|P_m'}(\nu) M_{P_m'|P}(\nu) \ \UUU \ \rho_{P, \nu}(\Delta^{2k})}\bigg|_{\AAA(X_P, \sigma)^{\tau, K}} \d \nu
		\end{multline*}
		which can be simplified to 
		\[ \sum_{\tau \in \hat \K_\infty} \sum_{\sigma \in \Pi_{\text{disc}}(M(\A))} \dim(\AAA(X_P, \sigma)^{\tau, K}) \int_{i (\aaa_{\tild L}^{\tild G})^*} \mod{\mu(\sigma, \nu, \tau)}^{-2k} \prod_{i=1}^m \norm{\delta_{P_i|P_i'}(\nu) \big|_{\AAA(X_P, \sigma)^{\tau, K}} } \d \nu. \]
		Now we can proceed according to \cite{FLM}*{\S5.1}. 
		
	\end{proof}


\section{An Application} \label{section:application}

In this section we discuss an application of the convergence of the spectral side to give a geometric expression for a sum of residues of poles of certain Rankin-Selberg $L$-functions, under the following assumption. 

\begin{assumption} \label{assump}
	The Root Cone Lemma holds for the cyclic base change case. 
\end{assumption}

Suppose $E/F$ is a Galois extension of number fields and $G = GL(n)$. Assume that the Galois group $\Gamma$ is generated by two elements $\theta$ and $\sigma$. For proving functoriality for base change this can be assumed without loss in generality, cf. \cite{MR2977357}*{\S 7}. If $w$ is finite, let $K_w = G(\mathcal O_{E_w})$ and $K_w = O(n, E_w)$ otherwise. Then $\K = \prod_w K_w$ is a ``good'' maximal compact subgroup in the sense of \cite{MR625344}. Denote the set of Archimedean places of $E$ by $\infty$.  For each $w \in \infty$, let $\phi_w$ be a smooth function on $\tild G(E_w)$ of compact support and bi-invariant under $K_w$, and set $\phi_\infty = \prod_{w \in \infty} \phi_w$. A cuspidal automorphic representation $\pi$ of $GL(n, \A_E)$ decomposes as $\pi = \otimes_w \pi_w$ and we can form the partial Rankin-Selberg $L$-function $L^\infty(s, \pi \times \tild \pi^\sigma)$ where $\tild \pi$ is the contragradient of $\pi$ and $\pi^\sigma(g) = \pi(g^\sigma)$. 

\begin{theorem}
	Suppose that \cref{assump} holds. With notations as above, we have that the sum of residues 
	\begin{equation} \label{eq:application}
		\sum_{\pi \simeq \pi^\Gamma} \prod_{w | \infty} \trace \pi_w(\phi_w) \Res_{s = 1} \displaystyle L^\infty(s, \pi \times \tild \pi^\sigma)
	\end{equation}
	is finite. The above sum is taken over cuspidal automorphic representations $\pi$ of $GL(n, \A_E)$ invariant under $\Gamma$ which are unramified at every finite place. 
\end{theorem}

\begin{remark}
The functorial transfer of automorphic representations attached to the $L$-homomorphism 
\[ b_{E/F} : {}^L\GL(n)_F \to {}^L\Res_{E/F}\GL(n)_{E} \]
when $\Gal(E/F)$ is cyclic of prime order is understood completely thanks to the work of Arthur and Clozel \cite{AC}. An important case is when $\Gal(E/F)$ is a universal perfect central extension of a simple nonabelian group, hence a quasi-simple group, see \cite{MR2977357}*{\S 7}. In this case by \cite{MR1800754}*{Corollary}, we can assume that $\Gal(E/F) = \sprod{\sigma}{\theta}$. For conjectural applications to Langlands' ``Beyond Endoscopy" program to approach Functoriality, it is important to understand residues of $L$-functions coming from various $L$-homomorphisms. The result at hand gives a geometric expression for such a sum of residues. 
\end{remark}

\begin{remark}
	Up to terms in the spectral expansion corresponding to proper Levi subgroups, the above sum of residues equals the geometric side of the twisted trace formula. This geometric interpretation, rather than absolute convergence is the contention of the theorem. Absolute convergence follows easily from that of the usual (non-twisted) trace formula. 
\end{remark}

\begin{proof}
	The convergence of the expression
	\[ \sum_{\pi \simeq \pi^\Gamma} \trace \pi_\infty(\phi_\infty) \displaystyle L^\infty(s, \pi \times \tild \pi^\sigma) \]
	for $\Re(s) \gg 0$ would follow by applying the convergence of the spectral side of the twisted trace formula to a certain basic function. However taking the residue at $s=1$ is a delicate question to prove which, will require the main result from \cite{Get15} about spherical Fourier transforms. 
	Let $F'$ be the fixed field of $\theta$ in $E$ and consider the group $G = \Res_{E/F'} GL(n)_E$. The automorphism $\theta$ acts on 
	\[ G(E) \cong GL(n, E) \times \cdots \times GL(n, E) \]
	(the number of copies being the order of $\theta$, say $\ell$) and we can form the semidirect product $G \rtimes \langle \theta \rangle$ of which $G \rtimes \theta$ is a coset. Since this automorphism $\theta$ preserves the Borel subgroup and torus, we can identify the group $G$ with the coset $G \rtimes \theta$ using the map $x \mapsto x \rtimes \theta$ and also identify functions on both cosets. We will construct the test function in $\CCC(\tild G(\A_E), K)$ which by the above identification and \cref{bijectionlemma} can be considered on the space $\CCC(G(\A_E), K)$. We will apply the spectral side convergence result to this function. 
	
	It is known that for every non-Archimedean place $w$, there exists a unique smooth function $\phi_{w, s}$ on $G(E_w)$, the basic function, such that 
	\begin{equation} \label{eq:L}
		\trace \pi_w(\phi_{w, s}) = L_w(s, \pi \times \tild \pi^\sigma) 
	\end{equation}
	holds for every irreducible admissible unramified representation $\pi_w$ of $G(E_w)$ and $\Re(s) \gg 0$. (See \cite{MR3220933} and \cite{MR2977357} for details.) Using the decomposition $\tild G(\A_E) = G(\A^\infty) \tild G(\A_\infty)$ define the function $\phi_s = \phi_\infty \prod_{w < \infty} \phi_{w, s}$. Then for $\Re(s)$ large enough depending on $n$, the function 
	\[ f_s(y) = \int_{A_G} \phi_s(ay) \d a, \qquad y \in \tild G(\A_E) \]
	as well as its Archimedean (spherical) Fourier transform (as defined in \cite{Get15})
	\[ \hat f_s(y) = \int_{A_G} (\mathcal F_{r, \psi}(\phi_\infty) \prod_{w < \infty} \phi_{w, s})(ay) \d a \]
	are elements of $\CCC(\tild G(\A_E), K)$. Here $\psi$ is the additive character used to define the Fourier transform and for the case at hand, $r$ can be taken to be the representation
	\[ r : {}^L{\Res_{E/F}} \GL(n) \to \GL(V)\]
	satisfying  
	\[ L(s, \pi, r) = L(s, \pi \times \pi^{\vee \sigma}). \]
	Note that the Fourier transform $\mathcal F_{r, \psi}(\phi_\infty)$ of the Archimedean component is in the space $\mathcal S^p(G(E_\infty) // K_\infty)$ of \cite{Get15}*{\S 3.3} for $p \in (0,1]$ so in particular satisfies the condition at the Archimedean component of our class $\mathcal C(\tild G(\A), K)$. If $\pi$ is a cuspidal automorphic representation of $G(\A_E)$ then its contribution to the spectral side of the trace formula for $\tild G$ will be nonzero precisely if it is invariant under $\theta$ and thanks to the choice of the test function, would then equal 
	\[ m(\pi) L^S(s, \pi \times \tild \pi^\sigma) \prod_{w \in S} \trace \pi_w(\phi_w). \]
	Although \Cref{eq:L} is valid for $\Re(s) \gg 0$, the completed Rankin-Selberg $L$-function $L(s, \pi \times \tild \pi^\sigma)$ is known to be entire when $\pi \not\simeq \pi^\sigma$. Clearly the residue at $s = 1$ in this case is zero. However if $\pi \simeq \pi^\sigma$ then $L(s, \pi \times \tilde \pi^\sigma)$ has meromorphic continuation to $\C$, satisfies a functional equation, has simple poles at $s=0, 1$ and no other poles \cite{MR701565}. Additionally by the multiplicity one theorem for $GL(n)$ \cite{MR0348047}, $m(\pi) = 1$. By \cite{Get15}*{Proposition 5.3}, the sum
	\[ \sum_\pi \trace \pi_\infty(\phi_\infty). \Res_{s=1} L^\infty(s, \pi \times \tilde \pi^\sigma) \]
	equals 
	\[ \frac{1}{2\pi i} \sum_\pi \int_{\Re(s) = \sigma} \trace \pi(f_s)\ \d s - \frac{1}{2\pi i} \sum_\pi \int_{\Re(s) = \sigma} \trace \pi(\hat f_s)\ \d s\]
	This sum is over cuspidal automorphic representations $\pi$ invariant under the Galois group $\Gal(E/F)$ and unramified at every finite place. Although this depends on [ibid, Conjecture 5.2], the functional equation is known for Rankin-Selberg $L$-functions. As explained above the two terms in the above difference are both bounded by the discrete part of the twisted trace formula with test functions satisfying \cref{eq:L} for $\Re(s) = \sigma$ sufficiently large. Thus their difference is finite. 
\end{proof}


\end{document}